\documentclass[]{article}

\usepackage[margin=1.5in]{geometry} 

\usepackage{latexsym}
\usepackage{amsmath}
\usepackage{amsfonts}
\usepackage{mathrsfs}
\usepackage{amstext}
\usepackage{amssymb}
\usepackage{amsthm}       

\numberwithin{equation}{section}

\newtheorem{theorem}{Theorem}[section]
\newtheorem{proposition}[theorem]{Proposition}

\newtheorem{definition}[theorem]{Definition}

\newtheorem{lemma}[theorem]{Lemma}

\begin{document}

\title{A singularly perturbed non-ideal transmission problem and application to
the effective conductivity of a periodic composite}

\author{Matteo Dalla Riva \footnote{M.~Dalla Riva acknowledges financial support from the Foundation for Science and Technology (FCT) via the post-doctoral grant SFRH/BPD/64437/2009. His work was supported also by {\it FEDER} funds through {\it COMPETE}--Operational Programme Factors of Competitiveness (``Programa Operacional Factores de Competitividade'') and by Portuguese funds through the {\it Center for Research and Development in Mathematics and Applications} (University of Aveiro) and the Portuguese Foundation for Science and Technology (``FCT--Funda\c{c}\~{a}o para a Ci\^{e}ncia e a Tecnologia''), within project PEst-C/MAT/UI4106/2011 with COMPETE number FCOMP-01-0124-FEDER-022690. }\and  Paolo Musolino \footnote{P.~Musolino acknowledges the financial support of the ``Fondazione Ing.~Aldo Gini''.}
}

\date{}

\maketitle

\noindent
{\bf Abstract:} We investigate the effective thermal conductivity of a two-phase composite with thermal resistance at the interface. The composite is obtained by introducing into an infinite homogeneous matrix a periodic set of inclusions of a different material. The diameter of each inclusion is assumed to be proportional to a positive real parameter $\epsilon$.  Under suitable assumptions, we show that the effective conductivity can be continued real analytically in the parameter $\epsilon$ around the degenerate value $\epsilon=0$, in correspondence of which the inclusions collapse to points.\footnote{Part of the results presented here have been announced in \cite{DaMu12a}.}  

\vspace{11pt}


\noindent
{\bf AMS:} 35J25; 31B10; 45A05; 74E30; 74G10; 74M15

\section{Introduction}
\label{introd}
We consider the heat conduction in a periodic two-phase composite with thermal resistance at the two-phase interface. The composite consists of a matrix and of a periodic set of inclusions. The matrix and the inclusions are filled with two (possibly different) homogeneous and isotropic heat conductor materials.  We assume that each inclusion  occupies a bounded domain of $\mathbb{R}^n$ of diameter proportional to a parameter $\epsilon>0$. The normal component of the heat flux is assumed to be continuous at the two-phase interface, while we impose that the temperature field displays a jump proportional to the normal heat flux by means of a parameter $\rho(\epsilon)>0$. In physics, the appearance of such a discontinuity in the temperature field is a well known  phenomenon and has been  largely investigated since 1941, when Kapitza carried out the first systematic study of thermal interface behaviour in liquid helium (see, {\it e.g.},  Swartz and Pohl \cite{SwPo89},  Lipton \cite{Li98} and references therein). The aim of this paper is to study the behaviour of the effective conductivity of the composite when the parameter $\epsilon$ tends to $0$ and the size of the inclusions collapses. The expression defining the effective conductivity of a composite with imperfect contact conditions was introduced by Benveniste and Miloh in \cite{BeMi86} by generalizing the dual theory of the effective behaviour of composites with perfect contact (see also Benveniste \cite{Be86} and for a review  Dryga{\'s} and Mityushev \cite{DrMi09}). By the argument of Benveniste and Miloh, in order to evaluate the effective conductivity, one has to study the thermal distribution of the composite when so called ``homogeneous conditions" are prescribed.  
To do so, we now introduce a particular transmission problem with non-ideal contact conditions where we impose that the temperature field displays a fixed jump along a certain direction and is periodic in all the other directions (cf.~problem \eqref{bvpe} below). For the sake of completeness, non-homogeneous boundary conditions at the two-phase interface are also investigated.

We fix once for all 
\[
n\in {\mathbb{N}}\setminus\{0,1 \}\,,\qquad (q_{11},\dots,q_{nn})\in]0,+\infty[^{n}\,.
\]
Then we introduce the periodicity cell $Q$ and the diagonal matrix $q$ by setting
\[
Q\equiv\Pi_{i=1}^{n}]0,q_{ii}[\,,\qquad q\equiv \left(\delta_{h,i}q_{ii}\right)_{(h,i)\in\{1,\dots,n\}^2}\,.
\]
Here $\delta_{h,i}\equiv 1$ if $h=i$ and $\delta_{h,i}\equiv 0$ if $h\neq i$.  We denote by $|Q|_n$ the $n$-dimensional measure of the fundamental cell $Q$ and by $q^{-1}$ the inverse matrix of $q$. Clearly, 
$
q {\mathbb{Z}}^{n}\equiv  \{qz:\,z\in{\mathbb{Z}}^{n}\}
$
is the set of vertices of a periodic subdivision of ${\mathbb{R}}^{n}$ corresponding to the fundamental cell $Q$.

Then we consider $\alpha\in]0,1[$ and a subset $\Omega$ of ${\mathbb{R}}^{n}$ satisfying the following assumption.
\begin{equation}
\begin{split}
&\text{$\Omega$ is a bounded open connected subset of $\mathbb{R}^{n}$ of  class  $C^{1,\alpha}$}
\\
&\text{such that $\mathbb{R}^{n}\setminus{\mathrm{cl}}\Omega$ is  connected   and   that   $0\in \Omega$.}\label{dom}
\end{split}
\end{equation}
The symbol `${\mathrm{cl}}$' denotes the closure. Let now $p\in Q$ be fixed. Then there exists $\epsilon_{0}\in\mathbb{R}$ such that
\begin{equation}
\label{e0}
\epsilon_{0}\in]0,+\infty[\,, \qquad p+\epsilon {\mathrm{cl}}\Omega\subseteq Q
\qquad\forall \epsilon\in]-\epsilon_{0},\epsilon_{0}[\,.
\end{equation}
To shorten our notation, we set
\[
\Omega_{p,\epsilon}\equiv p+\epsilon\Omega\qquad\forall\epsilon\in{\mathbb{R}}\,.
\]
Then we introduce the periodic domains
\[
{\mathbb{S}}[\Omega_{p,\epsilon}]\equiv\bigcup_{z\in {\mathbb{Z}}^{n}}\left(
qz+\Omega_{p,\epsilon}
\right)\,,
\qquad
{\mathbb{S}}[\Omega_{p,\epsilon}]^{-}\equiv
{\mathbb{R}}^{n}\setminus{\mathrm{cl}}{\mathbb{S}} [\Omega_{p,\epsilon}]\,,
\]
for all $\epsilon\in ]-\epsilon_{0},\epsilon_{0}[ $. 

Next, we take two positive constants $\lambda^+$, $\lambda^-$, a function $f$ in the Schauder space $C^{0,\alpha}(\partial \Omega)$ and with zero integral on $\partial \Omega$, a function $g$ in $C^{0,\alpha}(\partial \Omega)$, and a function $\rho$ from $]0,\epsilon_{0}[$ to $]0,+\infty[$, and for each $j \in \{1,\dots,n\}$ we consider the following transmission problem for a pair of functions $(u^+_j,u^-_j)\in C^{1,\alpha}_{\mathrm{loc}}(\mathrm{cl}\mathbb{S}[\Omega_{p,\epsilon}])\times C^{1,\alpha}_{\mathrm{loc}}(\mathrm{cl}\mathbb{S}[\Omega_{p,\epsilon}]^{-})$:
\begin{equation}
\label{bvpe}
\left\{
\begin{array}{ll}
\Delta u^+_j=0 & {\mathrm{in}}\ {\mathbb{S}}[\Omega_{p,\epsilon}]\,,\\
\Delta u^-_j=0 & {\mathrm{in}}\ {\mathbb{S}}[\Omega_{p,\epsilon}]^{-}\,,
\\
u^+_j(x+q_{hh}e_h)=u^+_j(x)+\delta_{h,j}q_{hh} & \forall x \in \mathrm{cl}\mathbb{S}[\Omega_{p,\epsilon}]\, ,\\
& \forall h \in\{1,\dots,n\}\,,\\
u^-_j(x+q_{hh}e_h)=u^-_j(x)+\delta_{h,j}q_{hh} & \forall x \in \mathrm{cl}\mathbb{S}[\Omega_{p,\epsilon}]^{-}\, ,\\
& \forall h \in\{1,\dots,n\}\,,\\
\lambda^-\frac{\partial u^-_j}{\partial\nu_{ \Omega_{p,\epsilon} }}(x)
-
\lambda^+\frac{\partial u^+_j}{\partial\nu_{ \Omega_{p,\epsilon} }}(x)=f((x-p)/\epsilon)& \forall x\in  \partial\Omega_{p,\epsilon}\,,\\
\lambda^+\frac{\partial u^+_j}{\partial\nu_{ \Omega_{p,\epsilon} }}(x)+\frac{1}{\rho(\epsilon)}\bigl(u^+_j(x)-u^-_j(x)\bigr)=g((x-p)/\epsilon)& \forall x\in  \partial\Omega_{p,\epsilon}\,,\\
\int_{\partial \Omega_{p,\epsilon}}u^+_j(x)\, d\sigma_x=0\, ,
\end{array}
\right.
\end{equation}
for all $\epsilon\in]0,\epsilon_{0}[$, 
where $\nu_{ \Omega_{p,\epsilon}}$ denotes the outward unit normal to $ \partial
\Omega_{p,\epsilon}$. Here $\{e_{1}$,\dots, $e_{n}\}$ denotes the canonical basis of ${\mathbb{R}}^{n}$. 

We observe that the functions $u^+_j$ and $u^-_j$ represent the temperature field in the inclusions occupying the periodic set ${\mathbb{S}}[\Omega_{p,\epsilon}]$ and in the matrix occupying ${\mathbb{S}}[\Omega_{p,\epsilon}]^{-}$, respectively.  The parameters $\lambda^+$ and $\lambda^-$ play the role of thermal conductivity of the materials which fill  the inclusions and the matrix, respectively. The parameter $\rho(\epsilon)$ plays the role of the interfacial thermal resistivity. The fifth condition in \eqref{bvpe} describes the jump of the normal heat flux across the two-phase interface  and the sixth condition describes the jump of the temperature field. In particular, if $f$ and $g$ are identically $0$, then the normal heat flux is continuous and the temperature field has a jump proportional to the  normal heat flux by means of the parameter $\rho(\epsilon)$. The third and fourth conditions in \eqref{bvpe} imply that the temperature distributions  $u^+_j$ and $u^-_j$ have a jump equal to $q_{jj}$ in the direction $e_j$ and are periodic in all the other directions. Finally, the seventh condition in \eqref{bvpe} is an auxiliary condition which we introduce  to guarantee the uniqueness of the solution $(u^+_j,u^-_j)$. Such a condition does not interfere in the definition of the effective conductivity, which is invariant for constant modifications of the temperature field.

We also observe that the boundary value problem in \eqref{bvpe} generalizes transmission problems which have been largely investigated in connection with the theory of heat conduction in two-phase periodic composites with imperfect contact conditions (cf., \textit{e.g.}, Castro and Pesetskaya \cite{CaPe10}, Castro, Pesetskaya, and Rogosin \cite{CaPeRo09},  Dryga{\'s} and Mityushev \cite{DrMi09}, Lipton \cite{Li98}, Mityushev \cite{Mi01}).

Due to the presence of the factor $1/\rho(\epsilon)$, the boundary condition may display a singularity as $\epsilon$ tends to $0$. In this paper, we consider the case in which
\begin{equation}
\label{varrhorast}
\text{the limit $\lim_{\epsilon\to 0^+}\frac{ \epsilon}{\rho(\epsilon)}$  exists  finite in  $\mathbb{R} $}.
\end{equation}
Assumption \eqref{varrhorast} will allow us to analyse problem \eqref{bvpe} around the degenerate value $\epsilon=0$. We also note that we make no regularity assumption on the function $\rho$. If assumption \eqref{varrhorast} holds, then we set
\begin{equation}\label{gast}
r_{\ast}\equiv\lim_{\epsilon\to 0^+}\frac{ \epsilon}{\rho(\epsilon)}\, .
\end{equation}

If $\epsilon \in ]0,\epsilon_0[$, then the solution in $C^{1,\alpha}_{\mathrm{loc}}(\mathrm{cl}\mathbb{S}[\Omega_{p,\epsilon}])\times C^{1,\alpha}_{\mathrm{loc}}(\mathrm{cl}\mathbb{S}[\Omega_{p,\epsilon}]^{-})$ of problem \eqref{bvpe} is unique and we denote it by $(u^+_j[\epsilon], u^-_j[\epsilon])$. Then we introduce the effective conductivity matrix $\lambda^{\mathrm{eff}}[\epsilon]$ with $(k,j)$-entry  $\lambda^{\mathrm{eff}}_{kj}[\epsilon]$ defined by means of the following.

\begin{definition}\label{def:S}
Let $\alpha\in]0,1[$. Let $p\in Q$. Let $\Omega$ be as in \eqref{dom}. Let $\epsilon_{0}$ be as in \eqref{e0}.  Let $\lambda^+, \lambda^- \in ]0,+\infty[$. Let $f,\,g \in C^{0,\alpha}(\partial \Omega)$ and $\int_{\partial\Omega}f\,d\sigma=0$. Let $\rho$ be a function from $]0,\epsilon_0[$ to $]0,+\infty[$. Let $(k,j) \in \{1,\dots,n\}^2$. We set
\[
\begin{split}
\lambda^{\mathrm{eff}}_{kj}[\epsilon]\equiv &\frac{1}{|Q|_n}\Biggl(\lambda^+\int_{\Omega_{p,\epsilon}}\frac{\partial u^+_j[\epsilon](x)}{\partial x_k} \, dx + \lambda^-\int_{Q\setminus \mathrm{cl}\Omega_{p,\epsilon}}\frac{\partial u^-_j[\epsilon](x)}{\partial x_k} \, dx \Biggr)\\
&+\frac{1}{|Q|_n}\int_{\partial\Omega_{p,\epsilon}}f((x-p)/\epsilon)x_k\, d\sigma_x\qquad\qquad\qquad\qquad\qquad\forall \epsilon \in ]0,\epsilon_0[\,,
\end{split}
\]
where $(u^+_j[\epsilon], u^-_j[\epsilon])$ is the unique solution in $C^{1,\alpha}_{\mathrm{loc}}(\mathrm{cl}\mathbb{S}[\Omega_{p,\epsilon}])\times C^{1,\alpha}_{\mathrm{loc}}(\mathrm{cl}\mathbb{S}[\Omega_{p,\epsilon}]^{-})$ of problem \eqref{bvpe}.
\end{definition}

We observe that Definition \ref{def:S} extends that of Benveniste and Miloh to the case of non-homogeneous boundary conditions and coincide with the classical definition  when $f$ and $g$ are identically $0$ (cf.~Benveniste \cite{Be86} and  Benveniste and Miloh \cite{BeMi86}). 

Next, if $(k,j)\in\{1,\dots,n\}^2$,  we pose the following question.
\begin{equation}\label{question}
\text{What  can be said on the map 
$\epsilon\mapsto \lambda^{\mathrm{eff}}_{kj} [\epsilon]$ when $\epsilon$ is close to $0$ and 
positive?} 
\end{equation}

Questions of this type have long been investigated with the methods of Asymptotic Analysis. 
Thus for example,  one could resort to Asymptotic Analysis and may succeed  to write out an asymptotic  expansion for $\lambda^{\mathrm{eff}}_{kj} [\epsilon]$. In this sense, we mention the works of Ammari and Kang \cite[Ch.~5]{AmKa07}, Ammari, Kang, and Touibi \cite{AmKaTo05}. We also mention Ammari, Kang, and Kim \cite{AmKaKi05} where the authors consider anisotropic heat conductors,  Ammari, Kang, and Lim \cite{AmKaLi06} where effective elastic properties are investigated, and Ammari, Garapon, Kang, and Lee \cite{AmGaKaLe} for the analysis of effective viscosity properties. For the application of asymptotic analysis to general elliptic problems we refer to Maz'ya, Nazarov, and Plamenewskij \cite{MaNaPl00} (see also  
Maz'ya, Movchan, and Nieves \cite{MaMoNi11} for mesoscale asymptotic approximations). For further references see, {\it e.g.}, Lanza de Cristoforis and the second author \cite{LaMu12}. Furthermore,  boundary value problems in domains with periodic inclusions have been  analysed, at least for the two dimensional case, with the method of functional equations (cf., {\it e.g.}, Castro and Pesetskaya \cite{CaPe10}, Castro, Pesetskaya, and Rogosin \cite{CaPeRo09},   Dryga{\'s} and Mityushev \cite{DrMi09}, Mityushev \cite{Mi01}).

Here we answer the question in \eqref{question} by showing that 
\[
\lambda^{\mathrm{eff}}_{kj} [\epsilon]=\lambda^-\delta_{k,j}+\epsilon^n\Lambda_{kj}[\epsilon,\epsilon/\rho(\epsilon)]
\] for $\epsilon>0$ small, where $\Lambda_{kj}$ is a real analytic map defined in a neighbourhood of the pair $(0,r_\ast)$.
We observe that our approach does have its advantages.  Indeed, if for 
example we know that $\epsilon/\rho(\epsilon)$ equals for $\epsilon>0$ a real analytic function defined in a whole neighbourhood of $\epsilon=0$, then we know that $\lambda^{\mathrm{eff}}_{kj}[\epsilon]$ can be expanded into a power series for $\epsilon$ small. This is the case if for example $\rho(\epsilon)=\epsilon$ or $\rho$ is constant. Such an approach has been carried out in the case of a simple hole, \textit{e.g.}, in Lanza de Cristoforis  \cite{La10} (see also \cite{DaMu12}), and has later been    extended to problems related to the system of equations of the linearized elasticity in \cite{DaLa10a} and to the Stokes system in \cite{Da11}, and to the case of problems  in an infinite periodically perforated domain in \cite{LaMu12, Mu12a}.

The paper is organized as follows. In \S \ref{nota}, we introduce some standard notation. In \S \ref{spaces}, \S \ref{plpotentials}, and \S \ref{prel} we show some preliminary results. In \S \ref{finteq}, we formulate our problem \eqref{bvpe} in terms of integral equations. In \S \ref{fure}, we investigate the asymptotic behaviour of $u^-_j[\epsilon]$ and $u^+_j[\epsilon]$. In \S \ref{eff}, we exploit the results of \S \ref{fure} to answer question \eqref{question}. Finally, in \S \ref{conrem}, we make some remarks and present some possible extensions of this work.

\section{Some notation}\label{nota}

We  denote the norm on 
a   normed space ${\mathcal X}$ by $\|\cdot\|_{{\mathcal X}}$. Let 
${\mathcal X}$ and ${\mathcal Y}$ be normed spaces. We endow the  
space ${\mathcal X}\times {\mathcal Y}$ with the norm defined by 
$\|(x,y)\|_{{\mathcal X}\times {\mathcal Y}}\equiv \|x\|_{{\mathcal X}}+
\|y\|_{{\mathcal Y}}$ for all $(x,y)\in  {\mathcal X}\times {\mathcal 
Y}$, while we use the Euclidean norm for ${\mathbb{R}}^{n}$.
The symbol ${\mathbb{N}}$ denotes the 
set of natural numbers including $0$.   If $A$ is a 
matrix, then  
 $A_{ij}$ denotes 
the $(i,j)$-entry of $A$. Let 
${\mathbb{D}}\subseteq {\mathbb {R}}^{n}$. Then $\mathrm{cl}{\mathbb{D}}$ 
denotes the 
closure of ${\mathbb{D}}$ and $\partial{\mathbb{D}}$ denotes the boundary of ${\mathbb{D}}$.  For all $R>0$, $ x\in{\mathbb{R}}^{n}$, 
$x_{j}$ denotes the $j$-th coordinate of $x$, 
$| x|$ denotes the Euclidean modulus of $ x$ in
${\mathbb{R}}^{n}$, and ${\mathbb{B}}_{n}( x,R)$ denotes the ball $\{
y\in{\mathbb{R}}^{n}:\, | x- y|<R\}$.
 
Let $\Omega$ be an open 
subset of ${\mathbb{R}}^{n}$. 
 Let $r\in {\mathbb{N}}\setminus\{0\}$.
Let $f\in \left(C^{m}(\Omega)\right)^{r}$. The 
$s$-th component of $f$ is denoted $f_{s}$, and $Df$ denotes the matrix
$\left(\frac{\partial f_s}{\partial
x_l}\right)_{  (s,l)\in \{1,\dots ,r\}\times \{1,\dots ,n\}       }$. 
For a multi-index  $\eta\equiv
(\eta_{1},\dots ,\eta_{n})\in{\mathbb{N}}^{n}$ we set $|\eta |\equiv
\eta_{1}+\dots +\eta_{n}  $. Then $D^{\eta} f$ denotes
$\frac{\partial^{|\eta|}f}{\partial
x_{1}^{\eta_{1}}\dots\partial x_{n}^{\eta_{n}}}$.    The
subspace of $C^{m}(\Omega )$ of those functions $f$ whose derivatives $D^{\eta }f$ of
order $|\eta |\leq m$ can be extended with continuity to 
$\mathrm{cl}\Omega$  is  denoted $C^{m}(
\mathrm{cl}\Omega )$. 
The
subspace of $C^{m}(\mathrm{cl}\Omega ) $  whose
functions have $m$-th order derivatives which are
uniformly H\"{o}lder continuous  with exponent  $\alpha\in
]0,1[$ is denoted $C^{m,\alpha} (\mathrm{cl}\Omega )$. The subspace of $C^{m}(\mathrm{cl}\Omega ) $ of those functions $f$ such that $f_{|{\mathrm{cl}}(\Omega\cap{\mathbb{B}}_{n}(0,R))}\in
C^{m,\alpha}({\mathrm{cl}}(\Omega\cap{\mathbb{B}}_{n}(0,R)))$ for all $R\in]0,+\infty[$ is denoted $C^{m,\alpha}_{{\mathrm{loc}}}(\mathrm{cl}\Omega ) $.  

Now let $\Omega $ be a bounded
open subset of  ${\mathbb{R}}^{n}$. Then $C^{m}(\mathrm{cl}\Omega)$ 
and $C^{m,\alpha }({\mathrm{cl}}
\Omega)$ are endowed with their usual norm and are well known to be 
Banach spaces.
We say that a bounded open subset $\Omega$ of ${\mathbb{R}}^{n}$ is of class 
$C^{m}$ or of class $C^{m,\alpha}$, if ${\mathrm{cl}}
\Omega$ is a 
manifold with boundary imbedded in 
${\mathbb{R}}^{n}$ of class $C^{m}$ or $C^{m,\alpha}$, respectively. We define the spaces $C^{k,\alpha}(\partial\Omega)$ for $k\in\{0,\dots,m\}$ by exploiting the local parametrizations  (cf., \textit{e.g.}, Gilbarg and Trudinger~\cite[\S 6.2]{GiTr83}). The trace operator from $C^{k,\alpha}({\mathrm{cl}}\Omega)$ to
$C^{k,\alpha}(\partial\Omega)$ is linear and continuous. 
 For standard properties of functions 
in Schauder spaces, we refer the reader to Gilbarg and 
Trudinger~\cite{GiTr83}
(see also  Lanza de Cristoforis \cite[\S 2, Lem.~3.1, 4.26, Thm.~4.28]{La91}, 
 Lanza de Cristoforis and Rossi \cite[\S 2]{LaRo04}).  
We denote by 
$
\nu_{\Omega}
$
the outward unit normal to $\partial\Omega$ and by $d\sigma$ the area element on $\partial\Omega$. We retain the standard notation for the Lebesgue space $L^{1}(\partial\Omega)$ of integrable functions. 

For the 
definition and properties of real analytic operators, we refer, {\it e.g.}, to Deimling \cite[p.~150]{De85}. In particular, we mention that the 
pointwise product in Schauder spaces is bilinear and continuous, and 
thus real analytic
(cf., \textit{e.g.}, Lanza de Cristoforis and Rossi \cite[pp.~141, 142]{LaRo04}).

\section{Spaces of bounded and periodic functions}\label{spaces}

 If $\Omega$ is an arbitrary open subset of ${\mathbb{R}}^{n}$, $k\in {\mathbb{N}}$, $\beta\in]0,1]$, we set
\[
C^{k}_{b}({\mathrm{cl}}\Omega)\equiv
\{
u\in C^{k}({\mathrm{cl}}\Omega):\,
D^{\gamma}u\ {\mathrm{is\ bounded}}\ \forall\gamma\in {\mathbb{N}}^{n}\
{\mathrm{such\ that}}\ |\gamma|\leq k
\}\,,
\]
and we endow $C^{k}_{b}({\mathrm{cl}}\Omega)$ with its usual  norm
\[
\|u\|_{ C^{k}_{b}({\mathrm{cl}}\Omega) }\equiv
\sum_{|\gamma|\leq k}\sup_{x\in {\mathrm{cl}}\Omega }|D^{\gamma}u(x)|\qquad\forall u\in C^{k}_{b}({\mathrm{cl}}\Omega)\,. 
\]
Then we set
\[
C^{k,\beta}_{b}({\mathrm{cl}}\Omega)\equiv
\{
u\in C^{k,\beta}({\mathrm{cl}}\Omega):\,
D^{\gamma}u\ {\mathrm{is\ bounded}}\ \forall\gamma\in {\mathbb{N}}^{n}\
{\mathrm{such\ that}}\ |\gamma|\leq k
\}\,,
\]
and we endow $C^{k,\beta}_{b}({\mathrm{cl}}\Omega)$ with its usual  norm
\[
\|u\|_{ C^{k,\beta}_{b}({\mathrm{cl}}\Omega) }\equiv
\sum_{|\gamma|\leq k}\sup_{x\in {\mathrm{cl}}\Omega }|D^{\gamma}u(x)|
+\sum_{|\gamma| = k}|D^{\gamma}u: {\mathrm{cl}}\Omega |_{\beta}
\qquad\forall u\in C^{k,\beta}_{b}({\mathrm{cl}}\Omega)\,,
\]
where $|D^{\gamma}u: {\mathrm{cl}}\Omega |_{\beta}$ denotes the $\beta$-H\"{o}lder constant of $D^{\gamma}u$.

Next we turn to periodic domains. If $\Omega_Q$ is an arbitrary subset of ${\mathbb{R}}^{n}$  such that
 ${\mathrm{cl}}\Omega_Q\subseteq Q$, then we set
\[
{\mathbb{S}} [\Omega_Q]\equiv 
\bigcup_{z\in{\mathbb{Z}}^{n} }(qz+\Omega_Q)=q{\mathbb{Z}}^{n}+\Omega_Q\,,
\qquad
{\mathbb{S}} [\Omega_Q]^{-}\equiv {\mathbb{R}}^{n}\setminus{\mathrm{cl}}{\mathbb{S}} [\Omega_Q]\,.
\]
 Then a function $u$  from ${\mathrm{cl}}{\mathbb{S}}[\Omega_Q]$ or from ${\mathrm{cl}}{\mathbb{S}}[\Omega_Q]^{-}$ to ${\mathbb{R}}$ 
is $q$-periodic if $u(x+q_{hh}e_{h})=u(x)$ for all $x$ in the domain of definition of $u$ and for all $h\in\{1,\dots,n\}$.
If $\Omega_Q$ is an open subset of ${\mathbb{R}}^{n}$  such that ${\mathrm{cl}}\Omega_Q\subseteq Q$ and if 
$k\in {\mathbb{N}}$ and $\beta\in]0,1[$, then we denote by $C^{k}_{q}({\mathrm{cl}}{\mathbb{S}}[\Omega_Q] )$, $C^{k,\beta}_{q}({\mathrm{cl}}{\mathbb{S}}[\Omega_Q] )$, $C^{k}_{q}({\mathrm{cl}}{\mathbb{S}}[\Omega_Q]^{-})$, and $C^{k,\beta}_{q}({\mathrm{cl}}{\mathbb{S}}[\Omega_Q]^{-} )$ the subsets of the $q$-periodic functions belonging to $C^{k}_{b}({\mathrm{cl}}{\mathbb{S}}[{\Omega_Q}])$, to $C^{k,\beta}_{b}({\mathrm{cl}}{\mathbb{S}}[{\Omega_Q}])$, to $C^{k}_{b}({\mathrm{cl}}{\mathbb{S}}[{\Omega_Q}]^- )$, and to $C^{k,\beta}_{b}({\mathrm{cl}}{\mathbb{S}}[{\Omega_Q}]^-)$, respectively. We regard the sets $C^{k}_{q}({\mathrm{cl}}{\mathbb{S}}[{\Omega_Q}])$, $C^{k,\beta}_{q}({\mathrm{cl}}{\mathbb{S}}[{\Omega_Q}])$, $C^{k}_{q}({\mathrm{cl}}{\mathbb{S}}[{\Omega_Q}]^-)$,  $C^{k,\beta}_{q}({\mathrm{cl}}{\mathbb{S}}[{\Omega_Q}]^-)$ as Banach subspaces of $C^{k}_{b}({\mathrm{cl}}{\mathbb{S}}[{\Omega_Q}])$, of $C^{k,\beta}_{b}({\mathrm{cl}}{\mathbb{S}}[{\Omega_Q}])$, of $C^{k}_{b}({\mathrm{cl}}{\mathbb{S}}[{\Omega_Q}]^-)$,  of $C^{k,\beta}_{b}({\mathrm{cl}}{\mathbb{S}}[{\Omega_Q}]^-)$, respectively.

\section{The periodic simple layer potential}\label{plpotentials}

As is well known there exists a $q$-periodic tempered distribution $S_{q,n}$ such that
\[
\Delta S_{q,n}=\sum_{z\in {\mathbb{Z}}^{n}}\delta_{qz}-\frac{1}{|Q|_n}\,,
\]
where $\delta_{qz}$ denotes the Dirac distribution with mass in $qz$. 
The distribution $S_{q,n}$ is determined  up to an additive constant, and we can take
\[
S_{q,n}(x)\equiv-\sum_{ z\in {\mathbb{Z}}^{n}\setminus\{0\} }
\frac{1}{     |Q|_n4\pi^{2}|q^{-1}z|^{2}   }e^{2\pi i (q^{-1}z)\cdot x}
\,,
\]
where the series converges in the sense of distributions on $\mathbb{R}^n$ (cf., \textit{e.g.}, Ammari and Kang~\cite[p.~53]{AmKa07},   
\cite[Theorems 3.1, 3.5]{LaMu11}). 
Then, $S_{q,n}$ is real analytic in ${\mathbb{R}}^{n}\setminus q{\mathbb{Z}}^{n}$ and is locally integrable in ${\mathbb{R}}^{n}$
 (cf., {\textit{e.g.}}, \cite[Theorem 3.5]{LaMu11}).

Let
$S_{n}$ be the function from ${\mathbb{R}}^{n}\setminus\{0\}$ to ${\mathbb{R}}$ defined by
\[
S_{n}(x)\equiv
\left\{
\begin{array}{lll}
\frac{1}{s_{n}}\log |x| \qquad &   \forall x\in 
{\mathbb{R}}^{n}\setminus\{0\},\quad & {\mathrm{if}}\ n=2\,,
\\
\frac{1}{(2-n)s_{n}}|x|^{2-n}\qquad &   \forall x\in 
{\mathbb{R}}^{n}\setminus\{0\},\quad & {\mathrm{if}}\ n>2\,,
\end{array}
\right.
\]
where $s_{n}$ denotes the $(n-1)$-dimensional measure of 
$\partial{\mathbb{B}}_{n}$. $S_{n}$ is well known to be the 
fundamental solution of the Laplace operator. 

Then  $S_{q,n}-S_{n}$ is analytic in $({\mathbb{R}}^{n}\setminus q{\mathbb{Z}}^{n})\cup\{0\}$ and we find convenient to set
\[
R_{q,n}\equiv S_{q,n}-S_{n}\qquad{\mathrm{in}}\ 
({\mathbb{R}}^{n}\setminus q{\mathbb{Z}}^{n})\cup\{0\}\,.
\]

We now introduce the classical simple layer potential. Let  $\alpha\in]0,1[$. Let $\Omega$ be a bounded open subset of ${\mathbb{R}}^{n}$ of class $C^{1,\alpha}$. Let $\mu\in C^{0,\alpha}(\partial\Omega)$. We set
\[
v[\partial\Omega,\mu](x)\equiv
\int_{\partial\Omega}S_{n}(x-y)\mu(y)\,d\sigma_{y}
\qquad\forall x\in {\mathbb{R}}^{n}\,.
\]
As is well known,  $v[\partial\Omega,\mu]$ is continuous in  ${\mathbb{R}}^{n}$, the function 
$v^{+}[\partial\Omega,\mu]\equiv v[\partial\Omega,\mu]_{|{\mathrm{cl}}\Omega}$ belongs to $C^{1,\alpha}({\mathrm{cl}}\Omega)$, and the function 
$v^{-}[\partial\Omega,\mu]\equiv v[\partial\Omega,\mu]_{|\mathbb{R}^n \setminus \Omega}$ belongs to $C^{1,\alpha}_{\mathrm{loc}}
(\mathbb{R}^n \setminus \Omega)$. Similarly,  we set
\[
w_{\ast}[\partial\Omega,\mu](x)\equiv
\int_{\partial\Omega}DS_{n}(x-y)\nu_{\Omega}(x)\mu(y)\,d\sigma_{y}
\qquad\forall x\in \partial \Omega\,.
\]
Then the function 
$w_{\ast}[\partial\Omega,\mu]$ belongs to $C^{0,\alpha}(\partial \Omega)$ and we have
\[
\frac{\partial }{\partial \nu_{\Omega}}v^\pm[\partial \Omega,\mu]=\mp \frac{1}{2}\mu + w_{\ast}[\partial \Omega,\mu]\qquad \text{on $\partial \Omega$}
\]
(cf., \textit{e.g.}, Miranda~\cite{Mi65},   Lanza de Cristoforis and Rossi \cite[Thm.~3.1]{LaRo04}). 

 If ${\mathcal{X}}$ is a vector subspace of $L^{1}(\partial\Omega )$, we find convenient to set
\[
{\mathcal{X}}_{0}\equiv
\left\{
f\in{\mathcal{X}}:\,\int_{\partial\Omega}f\,d\sigma=0 
\right\}\,.
\] 

Then we have the following well known result of classical potential theory. We note that statement (i) in Lemma \ref{lem0} below has been proved by Schauder \cite{Sc31,Sc32} for $n=3$, while the general case $n\geq 2$ is also valid and can be proved by slightly modifying the argument of Schauder \cite{Sc31,Sc32}. Statement (ii), instead, follows by Folland \cite[Prop.~3.19]{Fo95}, whereas statement (iii) by combining (i) and (ii).  

\begin{lemma}\label{lem0}  Let $\alpha\in]0,1[$. Let $\Omega$ be a bounded open subset of ${\mathbb{R}}^{n}$ of class $C^{1,\alpha}$. Then the following statements hold.
\begin{enumerate}
\item[(i)] The map from $C^{0,\alpha}(\partial \Omega)$ to $C^{0,\alpha}(\partial \Omega)$ which takes $\theta$ to $w_{\ast}[\partial \Omega,\theta]$ is compact.
\item[(ii)] If $\theta \in C^{0,\alpha}(\partial \Omega)_0$, then $w_{\ast}[\partial \Omega,\theta]\in C^{0,\alpha}(\partial \Omega)_0$.
\item[(iii)] The map from $C^{0,\alpha}(\partial \Omega)_0$ to $C^{0,\alpha}(\partial \Omega)_0$ which takes $\theta$ to $w_{\ast}[\partial \Omega,\theta]$ is compact.
\end{enumerate}
\end{lemma}

We now introduce the periodic simple layer potential. Let  $\alpha\in]0,1[$. Let $\Omega_Q$ be a bounded open subset of ${\mathbb{R}}^{n}$ of class $C^{1,\alpha}$ such that ${\mathrm{cl}}\Omega_Q\subseteq Q$. Let $\mu\in C^{0,\alpha}(\partial\Omega_Q)$. We set
\[
v_{q}[\partial\Omega_Q,\mu](x)\equiv
\int_{\partial\Omega_Q}S_{q,n}(x-y)\mu(y)\,d\sigma_{y}
\qquad\forall x\in {\mathbb{R}}^{n}\,.
\]
As is well known,  $v_q[\partial\Omega_Q,\mu]$ is continuous in  ${\mathbb{R}}^{n}$. Moreover, the function 
$v^{+}_{q}[\partial\Omega_Q,\mu]\equiv v_{q}[\partial\Omega_Q,\mu]_{|{\mathrm{cl}}{\mathbb{S}}[\Omega_Q]}$ belongs to $C^{1,\alpha}_{q}({\mathrm{cl}}{\mathbb{S}}[\Omega_Q])$, and  
$v^{-}_{q}[\partial\Omega_Q,\mu]\equiv v_{q}[\partial\Omega_Q,\mu]_{|{\mathrm{cl}}{\mathbb{S}}[\Omega_Q]^{-}}$ belongs to $C^{1,\alpha}_{q}
({\mathrm{cl}}{\mathbb{S}}[\Omega_Q]^{-})$. Similarly, we set
\[
w_{q,\ast}[\partial\Omega_Q,\mu](x)\equiv
\int_{\partial\Omega_Q}DS_{q,n}(x-y)\nu_{\Omega_Q}(x)\mu(y)\,d\sigma_{y}
\qquad\forall x\in \partial \Omega_Q\,.
\]
Then the function 
$w_{q,\ast}[\partial\Omega_Q,\mu]$ belongs to $C^{0,\alpha}(\partial \Omega_Q)$ and we have
\[
\frac{\partial }{\partial \nu_{\Omega_Q}}v^\pm_q[\partial \Omega_Q,\mu]=\mp \frac{1}{2}\mu + w_{q,\ast}[\partial \Omega_Q,\mu]\qquad \text{on $\partial \Omega_Q$}
\]
(cf., \textit{e.g.}, \cite[Theorem 3.7]{LaMu11}).

In the following lemma we have the periodic counterpart of Lemma \ref{lem0}.

\begin{lemma}\label{lem1}  Let  $\alpha\in]0,1[$. Let $\Omega_Q$ be a bounded open subset of ${\mathbb{R}}^{n}$ of class $C^{1,\alpha}$ such that ${\mathrm{cl}}\Omega_Q\subseteq Q$.  Then the following statements hold.
\begin{enumerate}
\item[(i)] The map from $C^{0,\alpha}(\partial \Omega_Q)$ to $C^{0,\alpha}(\partial \Omega_Q)$ which takes $\mu$ to $w_{q,\ast}[\partial \Omega_Q,\mu]$ is compact.
\item[(ii)] If $\mu \in C^{0,\alpha}(\partial \Omega_Q)_0$, then $w_{q,\ast}[\partial \Omega_Q,\mu]\in C^{0,\alpha}(\partial \Omega_Q)_0$.
\item[(iii)] The map from $C^{0,\alpha}(\partial \Omega_Q)_0$ to $C^{0,\alpha}(\partial \Omega_Q)_0$ which takes $\mu$ to $w_{q,\ast}[\partial \Omega_Q,\mu]$ is compact.
\end{enumerate}
\end{lemma}
\begin{proof}
We first consider (i). We note that
\begin{equation}\label{eq:lem1}
\begin{split}
&w_{q,\ast}[\partial\Omega_Q,\mu](x)\\
&\quad =w_{\ast}[\partial\Omega_Q,\mu](x)+\sum_{j=1}^n(\nu_{\Omega_Q}(x))_j\int_{\partial\Omega_Q}\partial_{x_j}R_{q,n}(x-y)\mu(y)\,d\sigma_y\quad \forall x\in\partial\Omega_Q\,,
\end{split}
\end{equation}
for all $\mu \in C^{0,\alpha}(\partial \Omega_Q)$. Then by the real analyticity of $\partial_{x_j}R_{q,n}$ in $(\mathbb{R}^n \setminus q\mathbb{Z}^n)\cup \{0\}$, by the compactenss of the imbedding of $C^{1,\alpha}(\partial \Omega_Q)$ into $C^{0,\alpha}(\partial \Omega_Q)$, by equality \eqref{eq:lem1}, and by Lemma \ref{lem0} (i), we deduce the validity of statement (i). Statement (ii) follows by Fubini's Theorem and by the well known identity
\[
\int_{\partial\Omega_Q}\frac{\partial}{\partial\nu_{\Omega_Q}(x)}
(S_{q,n}(y-x))\,d\sigma_{x}=\frac{1}{2}-\frac{|\Omega_Q|_n}{|Q|_n}
\qquad\forall y\in\partial{\mathbb{S}}[\Omega_Q]
\]
(cf., \textit{e.g.},  \cite[Lemma A.1]{LaMu12}). Here $|\Omega_Q|_n$ denotes the $n$-dimensional measure of $\Omega_Q$.  Finally, statement (iii) is a straightforward consequence of (i), (ii).\qquad\end{proof}

\section{Transmission problems with non-ideal contact conditions}\label{prel}

In this section we collect some preliminary results concerning transmission problems with non-ideal contact conditions.

We first have the following  uniqueness result for a periodic transmission problem, whose proof is based on a standard energy argument for periodic harmonic functions.

\begin{proposition}\label{prop:uniq}
Let $\alpha \in ]0,1[$. Let $\Omega_{Q}$ be a bounded open subset of $\mathbb{R}^n$ of class $C^{1,\alpha}$ such that $\mathbb{R}^n \setminus \mathrm{cl}\Omega_Q$ is connected and that $\mathrm{cl}\Omega_Q \subseteq Q$. Let $\lambda^+, \lambda^-, \gamma^\# \in ]0,+\infty[$. Let $(v^+,v^-) \in C^{1,\alpha}_q(\mathrm{cl}\mathbb{S}[\Omega_Q])\times C^{1,\alpha}_q(\mathrm{cl}\mathbb{S}[\Omega_Q]^{-})$ be such that
\begin{equation}
\label{eq:uniq}
\left\{
\begin{array}{ll}
\Delta v^+=0 & {\mathrm{in}}\ {\mathbb{S}}[\Omega_Q]\,,\\
\Delta v^-=0 & {\mathrm{in}}\ {\mathbb{S}}[\Omega_Q]^{-}\,,
\\
v^+(x+q_{hh}e_h)=v^+(x)& \forall x \in \mathrm{cl}\mathbb{S}[\Omega_Q]\, ,\ \forall h \in\{1,\dots,n\}\,,\\
v^-(x+q_{hh}e_h)=v^-(x) & \forall x \in \mathrm{cl}\mathbb{S}[\Omega_Q]^{-}\, ,\ \forall h \in\{1,\dots,n\}\,,\\
\lambda^-\frac{\partial v^-}{\partial\nu_{ \Omega_Q }}(x)
-
\lambda^+\frac{\partial v^+}{\partial\nu_{ \Omega_Q}}(x)=0& \forall x\in  \partial\Omega_Q\,,\\
\lambda^+\frac{\partial v^+}{\partial\nu_{\Omega_Q}}(x)+\gamma^\#\bigl(v^+(x)-v^-(x)\bigr)=0& \forall x\in  \partial\Omega_Q\,,\\
\int_{\partial \Omega_Q}v^+(x)\, d\sigma_x=0\, .
\end{array}
\right.
\end{equation}
Then $v^+=0$ on $\mathrm{cl}\mathbb{S}[\Omega_Q]$ and $v^-=0$ on $\mathrm{cl}\mathbb{S}[\Omega_Q]^-$.
\end{proposition}

We now  study an integral operator which we need in order to solve a periodic transmission problem by means of periodic simple layer potentials. 

\begin{proposition}\label{prop:J}
Let $\alpha \in ]0,1[$. Let $\Omega_{Q}$ be a bounded open subset of $\mathbb{R}^n$ of class $C^{1,\alpha}$ such that $\mathbb{R}^n \setminus \mathrm{cl}\Omega_Q$ is connected and that $\mathrm{cl}\Omega_Q \subseteq Q$. Let $\lambda^+, \lambda^-, \gamma^\# \in ]0,+\infty[$. Let $J_{\gamma^\#}\equiv(J_{\gamma^\#,1},J_{\gamma^\#,2})$ be the operator from $(C^{0,\alpha}(\partial \Omega_Q)_0)^2$ to $(C^{0,\alpha}(\partial\Omega_Q)_0)^2$ defined by
\begin{align}
J_{\gamma^\#,1}[\mu^i,\mu^o]\equiv&\lambda^-\Bigl(\frac{1}{2}\mu^o+w_{q,\ast}[\partial \Omega_Q,\mu^o]\Bigr)-\lambda^+\Bigl(-\frac{1}{2}\mu^i+w_{q,\ast}[\partial \Omega_Q,\mu^i]\Bigr)\, ,\nonumber\\
J_{\gamma^\#,2}[\mu^i,\mu^o]\equiv&\lambda^+\Bigl(-\frac{1}{2}\mu^i+w_{q,\ast}[\partial \Omega_Q,\mu^i]\Bigr)\nonumber\\&+\gamma^\# \Bigl(v^+_q[\partial \Omega_Q,\mu^i]_{|\partial\Omega_Q}-\frac{1}{|\partial\Omega_Q|_{n-1}}\int_{\partial \Omega_Q}v^+_q[\partial \Omega_Q,\mu^i]\, d\sigma\nonumber\\
&-v^-_q[\partial \Omega_Q,\mu^o]_{|\partial\Omega_Q}+\frac{1}{|\partial\Omega_Q|_{n-1}}\int_{\partial \Omega_Q}v^-_q[\partial \Omega_Q,\mu^o]\, d\sigma\Bigr)\, ,\nonumber
\end{align}
for all $(\mu^i,\mu^o) \in (C^{0,\alpha}(\partial \Omega_Q)_0)^2$, where $|\partial\Omega_Q|_{n-1}$ denotes the $(n-1)$-dimensional measure of $\partial\Omega_Q$. Then $J_{\gamma^\#}$ is a linear homeomorphism.
\end{proposition}
\begin{proof}
Let $\hat{J}_{\gamma^\#}\equiv(\hat{J}_{\gamma^\#,1},\hat{J}_{\gamma^\#,2})$ be the linear operator from $(C^{0,\alpha}(\partial\Omega_Q)_0)^2$ to $(C^{0,\alpha}(\partial\Omega_Q)_0)^2$ defined by 
\[
\begin{split}
&\hat{J}_{\gamma^\#,1}[\mu^i,\mu^o]\equiv (\lambda^-/2)\mu^o+(\lambda^+/2)\mu^i\,,\quad \hat{J}_{\gamma^\#,2}[\mu^i,\mu^o]\equiv -(\lambda^+/2)\mu^i
\end{split}
\]  for all $(\mu^i,\mu^o)\in(C^{0,\alpha}(\partial\Omega_Q)_0)^2$. 
Clearly, $\hat{J}_{\gamma^\#}$ is a linear homeomorphism from $(C^{0,\alpha}(\partial\Omega_Q)_0)^2$ to $(C^{0,\alpha}(\partial\Omega_Q)_0)^2$. Then let $\tilde{J}_{\gamma^\#}\equiv(\tilde{J}_{\gamma^\#,1},\tilde{J}_{\gamma^\#,2})$ be the operator from $(C^{0,\alpha}(\partial\Omega_Q)_0)^2$ to $(C^{0,\alpha}(\partial\Omega_Q)_0)^2$ defined by 
\[
\begin{split}
\tilde{J}_{\gamma^\#,1}[\mu^i,\mu^o]\equiv& \lambda^- w_{q,\ast}[\partial\Omega_Q,\mu^o]-\lambda^+ w_{q,\ast}[\partial\Omega_Q,\mu^i]\,,\\
\tilde{J}_{\gamma^\#,2}[\mu^i,\mu^o]\equiv &\lambda^+ w_{q,\ast}[\partial\Omega_Q,\mu^i]\\&+\gamma^\# \Bigl(v^+_q[\partial \Omega_Q,\mu^i]_{|\partial\Omega_Q}-\frac{1}{|\partial\Omega_Q|_{n-1}}\int_{\partial \Omega_Q}v^+_q[\partial \Omega_Q,\mu^i]\, d\sigma\\&-v^-_q[\partial \Omega_Q,\mu^o]_{|\partial\Omega_Q}+\frac{1}{|\partial\Omega_Q|_{n-1}}\int_{\partial \Omega_Q}v^-_q[\partial \Omega_Q,\mu^o]\, d\sigma\Bigr) 
\end{split}
\] for all $(\mu^i,\mu^o)\in(C^{0,\alpha}(\partial\Omega_Q)_0)^2$. 
Then, by Lemma~\ref{lem1},  by the boundedness of the operator from $C^{0,\alpha}(\partial\Omega_Q)_0$ to $C^{1,\alpha}(\partial\Omega_Q)_0$ which takes $\mu$ to 
\[
v_{q}[\partial\Omega_Q,\mu]_{|\partial \Omega_Q}-\frac{1}{|\partial\Omega_Q|_{n-1}}\int_{\partial \Omega_Q}v_{q}[\partial\Omega_Q,\mu]\, d\sigma\, ,
\] 
and by the compactness of the imbedding of $C^{1,\alpha}(\partial\Omega_Q)_0$ into $C^{0,\alpha}(\partial\Omega_Q)_0$, we have that $\tilde{J}_{\gamma^\#}$ is a compact operator. Now, since $J_{\gamma^\#}=\hat{J}_{\gamma^\#}+\tilde{J}_{\gamma^\#}$ and since compact perturbations of isomorphisms are Fredholm operators of index $0$, we deduce that $J_{\gamma^\#}$ is a Fredholm operator of index $0$. Thus to show that  $J_{\gamma^\#}$ is a linear homeomorphism it suffices to show that it is injective. So, let $(\mu^i,\mu^o)\in(C^{0,\alpha}(\partial\Omega_Q)_0)^2$ be such that $J_{\gamma^\#}[\mu^i,\mu^o]=(0,0)$. Then by the jump formulae for the normal derivative of the periodic simple layer potential one verifies that the pair $(v^+,v^-)\in C^{1,\alpha}_q(\mathrm{cl}\mathbb{S}[\Omega_Q])\times C^{1,\alpha}_q(\mathrm{cl}\mathbb{S}[\Omega_Q]^-)$ defined by
\[
\begin{split}
& v^+\equiv v^+_q[\partial \Omega_Q,\mu^i]-\frac{1}{|\partial\Omega_Q|_{n-1}}\int_{\partial \Omega_Q}v^+_q[\partial \Omega_Q,\mu^i]\, d\sigma\, ,\\
& v^-\equiv v^-_q[\partial \Omega_Q,\mu^o]-\frac{1}{|\partial\Omega_Q|_{n-1}}\int_{\partial \Omega_Q}v^-_q[\partial \Omega_Q,\mu^o]\, d\sigma\, ,
\end{split}
\] 
is a solution of the boundary value problem in \eqref{eq:uniq}. Accordingly, Proposition~\ref{prop:uniq} implies that $v^-=0$ and $v^+=0$. In particular,
\begin{align}
& v^+_q[\partial \Omega_Q,\mu^i]-\frac{1}{|\partial\Omega_Q|_{n-1}}\int_{\partial \Omega_Q}v^+_q[\partial \Omega_Q,\mu^i]\, d\sigma=0\qquad \text{in $\mathrm{cl}\mathbb{S}[\Omega_Q]$}\,,\label{120612eq1}\\
& v^-_q[\partial \Omega_Q,\mu^o]-\frac{1}{|\partial\Omega_Q|_{n-1}}\int_{\partial \Omega_Q}v^-_q[\partial \Omega_Q,\mu^o]\, d\sigma=0\qquad \text{in $\mathrm{cl}\mathbb{S}[\Omega_Q]^-$}\,.\nonumber
\end{align}
Then, by the jump formulae for the normal derivative of the periodic simple layer potential we have 
\[
\frac{\partial v^-}{\partial\nu_{\Omega_Q}}(x)=\frac{1}{2}\mu^o(x)+w_{q,\ast}[\partial\Omega_Q,\mu^o](x)=0\qquad\forall x\in\partial\Omega_Q\, ,
\] 
which implies that $\mu^o=0$ (cf.~\cite[Proposition A.4 (i)]{LaMu12}). Moreover, by \eqref{120612eq1} and by the continuity of the simple layer potential we deduce that
\[
v^-_q[\partial \Omega_Q,\mu^i]-\frac{1}{|\partial\Omega_Q|_{n-1}}\int_{\partial \Omega_Q}v^-_q[\partial \Omega_Q,\mu^i]\, d\sigma=0\qquad \text{in $\mathrm{cl}\mathbb{S}[\Omega_Q]^-$}\, ,
\]
and thus by arguing as above we conclude that $\mu^i=0$. Accordingly, $J_{\gamma^\#}$ is injective, and, as a consequence, a linear homeomorphism.
\qquad\end{proof}

By the jump formulae for the normal derivative of the periodic simple layer potential, we can now deduce the validity of the following theorem.

\begin{theorem}\label{thm:ex}
Let $\alpha \in ]0,1[$. Let $\Omega_{Q}$ be a bounded open subset of $\mathbb{R}^n$ of class $C^{1,\alpha}$ such that $\mathbb{R}^n \setminus \mathrm{cl}\Omega_Q$ is connected and that $\mathrm{cl}\Omega_Q \subseteq Q$. Let $\lambda^+, \lambda^-, \gamma^\# \in ]0,+\infty[$. Let $(\Phi,\Gamma,c) \in C^{0,\alpha}(\partial \Omega_Q)_0\times C^{0,\alpha}(\partial\Omega_Q)\times \mathbb{R}$. Let $J_{\gamma^\#}$ be as in Proposition \ref{prop:J}. Then a pair $(\mu^i,\mu^o) \in (C^{0,\alpha}(\partial \Omega_Q)_0)^2$ satisfies the equality 
\[
J_{\gamma^\#}[\mu^i,\mu^o]=\biggl(\Phi,\Gamma-\frac{\int_{\partial \Omega_Q}\Gamma\, d\sigma}{|\partial\Omega_Q|_{n-1}}\biggr)
\] 
if and only if the pair $(v^+,v^-) \in C^{1,\alpha}_q(\mathrm{cl}\mathbb{S}[\Omega_Q])\times C^{1,\alpha}_q(\mathrm{cl}\mathbb{S}[\Omega_Q]^-)$ defined by
\begin{align}
&v^+\equiv v^+_q[\partial \Omega_Q,\mu^i]-\frac{1}{|\partial\Omega_Q|_{n-1}}\int_{\partial \Omega_Q}v^+_q[\partial \Omega_Q,\mu^i]\, d\sigma+\frac{1}{|\partial\Omega_Q|_{n-1}}c\,, \nonumber\\
& v^-\equiv v^-_q[\partial \Omega_Q,\mu^o]-\frac{1}{|\partial\Omega_Q|_{n-1}}\int_{\partial \Omega_Q}v^-_q[\partial \Omega_Q,\mu^o]\, d\sigma\nonumber\\
&\qquad+\frac{1}{|\partial\Omega_Q|_{n-1}}c-\frac{1}{\gamma^\#}\frac{1}{|\partial\Omega_Q|_{n-1}}\int_{\partial \Omega_Q}\Gamma\,d\sigma\, ,\nonumber
\end{align}
is a solution of
\begin{equation}
\label{eq:ex2}
\left\{
\begin{array}{ll}
\Delta v^+=0 & {\mathrm{in}}\ {\mathbb{S}}[\Omega_Q]\,,\\
\Delta v^-=0 & {\mathrm{in}}\ {\mathbb{S}}[\Omega_Q]^{-}\,,
\\
v^+(x+q_{hh}e_h)=v^+(x)& \forall x \in \mathrm{cl}\mathbb{S}[\Omega_Q]\, ,\ \forall h \in\{1,\dots,n\}\,,\\
v^-(x+q_{hh}e_h)=v^-(x) & \forall x \in \mathrm{cl}\mathbb{S}[\Omega_Q]^{-}\, ,\ \forall h \in\{1,\dots,n\}\,,\\
\lambda^-\frac{\partial v^-}{\partial\nu_{ \Omega_Q }}(x)-\lambda^+\frac{\partial v^+}{\partial\nu_{ \Omega_Q}}(x)=\Phi(x)& \forall x\in  \partial\Omega_Q\,,\\
\lambda^+\frac{\partial v^+}{\partial\nu_{ \Omega_Q}}(x)+\gamma^\#\bigl(v^+(x)-v^-(x)\bigr)=\Gamma(x)& \forall x\in  \partial\Omega_Q\,,\\
\int_{\partial \Omega_Q}v^+(x)\, d\sigma_x=c\, .
\end{array}
\right.
\end{equation}
\end{theorem}

{\em Remark}. Let $\alpha$, $\Omega_Q$, $\lambda^+$, $\lambda^-$, $\gamma^\#$ be as in Theorem \ref{thm:ex}. Then Propositions \ref{prop:uniq}, \ref{prop:J} and Theorem \ref{thm:ex} imply that for each triple  $(\Phi,\Gamma,c) \in C^{0,\alpha}(\partial \Omega_Q)_0\times C^{0,\alpha}(\partial\Omega_Q)\times \mathbb{R}$, the solution in $C^{1,\alpha}_q(\mathrm{cl}\mathbb{S}[\Omega_Q])\times C^{1,\alpha}_q(\mathrm{cl}\mathbb{S}[\Omega_Q]^-)$ of problem \eqref{eq:ex2} exists and is unique. 

We now turn to non-periodic problems and we prove some results which we use in the sequel  to analyse problem \eqref{bvpe} around the degenerate case $\epsilon=0$. We first have the following uniqueness result whose validity can be deduced by  a standard energy argument.

\begin{proposition}\label{prop:uniqbis}
Let $\alpha \in ]0,1[$. Let $\Omega$ be a bounded open connected subset of $\mathbb{R}^n$ of class $C^{1,\alpha}$ such that $\mathbb{R}^n \setminus \mathrm{cl}\Omega$ is connected. Let $\lambda^+, \lambda^- \in ]0,+\infty[$, $\tilde{\gamma} \in [0,+\infty[$. Let $(v^+,v^-) \in C^{1,\alpha}(\mathrm{cl}\Omega)\times C^{1,\alpha}_{\mathrm{loc}}(\mathbb{R}^n \setminus  \Omega)$ be such that
\begin{equation}
\label{eq:uniqbis}
\left\{
\begin{array}{ll}
\Delta v^+=0 & {\mathrm{in}}\ \Omega\,,\\
\Delta v^-=0 & {\mathrm{in}}\ \mathbb{R}^n \setminus \mathrm{cl}\Omega\,,
\\
\lambda^-\frac{\partial v^-}{\partial\nu_{ \Omega}}(x)-\lambda^+\frac{\partial v^+}{\partial\nu_{ \Omega}}(x)=0& \forall x\in  \partial\Omega\,,\\
\lambda^+\frac{\partial v^+}{\partial\nu_{ \Omega}}(x)+\tilde{\gamma}\Bigl(v^+(x)-v^-(x)\Bigr)=0& \forall x\in  \partial\Omega\,,\\
\int_{\partial \Omega}v^+(x)\, d\sigma_x=0\, ,\\
\int_{\partial \Omega}v^-(x)\, d\sigma_x=0\, ,\\
\lim_{x\to \infty}v^-(x) \in \mathbb{R}\, .
\end{array}
\right.
\end{equation}
Then $v^+=0$ on $\mathrm{cl}\Omega$ and $v^-=0$ on $\mathbb{R}^n \setminus  \Omega$.
\end{proposition}

We now  study an integral operator which we need in order to solve (non-periodic) transmission problems in terms of classical simple layer potentials.

\begin{proposition}\label{prop:K}
Let $\alpha \in ]0,1[$. Let $\Omega$ be a bounded open connected subset of $\mathbb{R}^n$ of class $C^{1,\alpha}$ such that $\mathbb{R}^n \setminus \mathrm{cl}\Omega$ is connected. Let $\lambda^+, \lambda^- \in ]0,+\infty[$, $\tilde{\gamma} \in [0,+\infty[$. Let $K_{\tilde{\gamma}}\equiv(K_{\tilde{\gamma},1},K_{\tilde{\gamma},2})$ be the operator from $(C^{0,\alpha}(\partial \Omega)_0)^2$ to $(C^{0,\alpha}(\partial\Omega)_0)^2$ defined by
\begin{align*}
K_{\tilde{\gamma},1}[\theta^i,\theta^o]\equiv&\lambda^-\Bigl(\frac{1}{2}\theta^o+w_{\ast}[\partial \Omega,\theta^o]\Bigr)-\lambda^+\Bigl(-\frac{1}{2}\theta^i+w_{\ast}[\partial \Omega,\theta^i]\Bigr)\, ,\nonumber\\
K_{\tilde{\gamma},2}[\theta^i,\theta^o]\equiv&\lambda^+\Bigl(-\frac{1}{2}\theta^i+w_{\ast}[\partial \Omega,\theta^i]\Bigr)+\tilde{\gamma} \biggl(v^+[\partial \Omega,\theta^i]_{|\partial \Omega}-\frac{1}{|\partial\Omega|_{n-1}}\int_{\partial \Omega}v^+[\partial \Omega,\theta^i]\, d\sigma\\&-v^-[\partial \Omega,\theta^o]_{|\partial \Omega}+\frac{1}{|\partial\Omega|_{n-1}}\int_{\partial \Omega}v^-[\partial \Omega,\theta^o]\, d\sigma\biggr)\, ,\nonumber
\end{align*}
for all $(\theta^i,\theta^o) \in (C^{0,\alpha}(\partial \Omega)_0)^2$, where $|\partial\Omega|_{n-1}$ denotes the $(n-1)$-dimensional measure of $\partial\Omega$. Then $K_{\tilde{\gamma}}$ is a linear homeomorphism.
\end{proposition}
\begin{proof}
By arguing so as in the proof of Proposition \ref{prop:J} for $J_{\gamma^\#}$ and by replacing Lemma \ref{lem1} by Lemma \ref{lem0}, one can prove that $K_{\tilde{\gamma}}$ is a Fredholm operator of index $0$. Thus to show that  $K_{\tilde{\gamma}}$ is a linear homeomorphism it suffices to show that it is injective. So, let $(\theta^i,\theta^o)\in(C^{0,\alpha}(\partial\Omega)_0)^2$ be such that $K_{\tilde{\gamma}}[\theta^i,\theta^o]=(0,0)$. Then by the jump formulae for the normal derivative of the simple layer potential, the pair $(v^+,v^-)\in C^{1,\alpha}(\mathrm{cl}\Omega)\times C^{1,\alpha}_{\mathrm{loc}}(\mathbb{R}^n \setminus\Omega)$ defined by
\[
\begin{split}
&v^+\equiv v^+[\partial \Omega,\theta^i]-\frac{1}{|\partial\Omega|_{n-1}}\int_{\partial \Omega}v^+[\partial \Omega,\theta^i]\, d\sigma\, ,\\
& v^-\equiv v^-[\partial \Omega,\theta^o]-\frac{1}{|\partial\Omega|_{n-1}}\int_{\partial \Omega}v^-[\partial \Omega,\theta^o]\, d\sigma\, ,
\end{split}
\] 
is a solution of the boundary value problem in \eqref{eq:uniqbis}. Accordingly, Proposition~\ref{prop:uniqbis} implies that $v^-=0$ and $v^+=0$. Then, by classical potential theory, $\theta^i=0$ and $\theta^o=0$ (cf., \textit{e.g.}, Folland \cite[Chapter 3, \S D]{Fo95}).  Accordingly, $K_{\tilde{\gamma}}$ is injective, and, as a consequence, a linear homeomorphism.
\end{proof}

By Propositions \ref{prop:uniqbis}, \ref{prop:K}, and by the jump formulae for the normal derivative of the classical simple layer potential, we immediately deduce the validity of the following result concerning the solvability of a (non-periodic) transmission problem.

\begin{theorem}\label{thm:exbis}
Let $\alpha \in ]0,1[$. Let $\Omega$ be a bounded open connected subset of $\mathbb{R}^n$ of class $C^{1,\alpha}$ such that $\mathbb{R}^n \setminus \mathrm{cl}\Omega$ is connected. Let $\lambda^+, \lambda^- \in ]0,+\infty[$, $\tilde{\gamma} \in [0,+\infty[$. Let $K_{\tilde{\gamma}}$ be as in Proposition \ref{prop:K}. Let $(\Phi,\Gamma) \in (C^{0,\alpha}(\partial \Omega)_0)^2$. Let $(\theta^i,\theta^o) \in (C^{0,\alpha}(\partial \Omega)_0)^2$ be such that
\[
K_{\tilde{\gamma}}[\theta^i,\theta^o]=(\Phi,\Gamma)\,.
\] 
Let $(v^+,v^-)\in C^{1,\alpha}(\mathrm{cl}\Omega)\times C^{1,\alpha}_{\mathrm{loc}}(\mathbb{R}^n \setminus \Omega)$ be defined by
\begin{align}
&v^+\equiv v^+[\partial\Omega,\theta^i]-\frac{1}{|\partial\Omega|_{n-1}}\int_{\partial \Omega}v^+[\partial\Omega,\theta^i]\,d\sigma\,,\nonumber\\
& v^-\equiv v^-[\partial\Omega,\theta^o]-\frac{1}{|\partial\Omega|_{n-1}}\int_{\partial \Omega}v^-[\partial\Omega,\theta^o]\,d\sigma\, . \nonumber
\end{align}
Then $(v^+,v^-)$ is the unique solution in $C^{1,\alpha}(\mathrm{cl}\Omega)\times C^{1,\alpha}_{\mathrm{loc}}(\mathbb{R}^n \setminus \Omega)$ of 
\[
\left\{
\begin{array}{ll}
\Delta v^+=0 & {\mathrm{in}}\ \Omega\,,\\
\Delta v^-=0 & {\mathrm{in}}\ \mathbb{R}^n \setminus \mathrm{cl}\Omega\,,
\\
\lambda^-\frac{\partial v^-}{\partial\nu_{ \Omega}}(x)-\lambda^+\frac{\partial v^+}{\partial\nu_{ \Omega}}(x)=\Phi(x)& \forall x\in  \partial\Omega\,,\\
\lambda^+\frac{\partial v^+}{\partial\nu_{ \Omega}}(x)+\tilde{\gamma}\Bigl(v^+(x)-v^-(x)\Bigr)=\Gamma(x)& \forall x\in  \partial\Omega\,,\\
\int_{\partial \Omega}v^+(x)\, d\sigma_x=0\, ,\\
\int_{\partial \Omega}v^-(x)\, d\sigma_x=0\, ,\\
\lim_{x\to \infty}v^-(x) \in \mathbb{R}\, .
\end{array}
\right.
\]
\end{theorem}
\section{Formulation of  problem \eqref{bvpe} in terms of integral equations}\label{finteq}
In the following Proposition \ref{prop:finteq}, we formulate problem \eqref{bvpe} in terms of integral equations on $\partial \Omega$. To do so, we exploit Theorem \ref{thm:ex} and the rule of change of variables in integrals. Indeed, if $\epsilon \in ]0,\epsilon_0[$, by a simple computation one can convert problem \eqref{bvpe} into a periodic transmission problem (see problem \eqref{inteq3} below). Then, by Theorem \ref{thm:ex}, one can reformulate such a problem in terms of a system of integral equations defined on the $\epsilon$-dependent domain $\partial \Omega_{p,\epsilon}$. Finally, by exploiting an appropriate change of variable, one can get rid of such a dependence and can obtain an equivalent system of integral equations defined on the fixed domain $\partial \Omega$, as the following proposition shows.   We now find convenient to introduce the following notation. Let $\alpha\in]0,1[$. Let $\Omega$ be as in \eqref{dom}. Let $\epsilon_{0}$ be as in \eqref{e0}. If $\lambda^+, \lambda^- \in ]0,+\infty[$, $f \in C^{0,\alpha}(\partial \Omega)_0$, $g \in C^{0,\alpha}(\partial \Omega)$, then we denote by $M \equiv (M_1,M_2)$ the operator from $]-\epsilon_{0},\epsilon_{0}[\times \mathbb{R} \times (C^{0,\alpha}(\partial\Omega)_{0})^2$ to $(C^{0,\alpha}(\partial\Omega)_0)^2$ defined by
\begin{equation}\label{Lmbdeq1}
\begin{split}
M_1&[\epsilon,\epsilon',\theta^i,\theta^o](t)\\&\equiv\lambda^-\Bigl(\frac{1}{2}\theta^o(t)+w_{\ast}[\partial \Omega,\theta^o](t)+\epsilon^{n-1}\int_{\partial \Omega}DR_{q,n}(\epsilon(t-s))\nu_{\Omega}(t)\theta^o(s)\, d\sigma_s\Bigr)\\
&-\lambda^+\Bigl(-\frac{1}{2}\theta^i(t)+w_{\ast}[\partial \Omega,\theta^i](t)+\epsilon^{n-1}\int_{\partial \Omega}DR_{q,n}(\epsilon(t-s))\nu_{\Omega}(t)\theta^i(s)\, d\sigma_s\Bigr)\\
&- f(t)+(\lambda^--\lambda^+)(\nu_{\Omega}(t))_j\qquad\qquad\qquad\qquad\qquad\qquad \forall t \in \partial \Omega\, ,
\end{split}
\end{equation}
\begin{equation}\label{Lmbdeq2}
\begin{split}
M_2&[\epsilon,\epsilon',\theta^i,\theta^o](t)\\&\equiv\lambda^+\Bigl(-\frac{1}{2}\theta^i(t)+w_{\ast}[\partial \Omega,\theta^i](t)+\epsilon^{n-1}\int_{\partial \Omega}DR_{q,n}(\epsilon(t-s))\nu_{\Omega}(t)\theta^i(s)\, d\sigma_s\Bigr)\\
&+ \epsilon' \Biggl(v^+[\partial \Omega,\theta^i](t)+\epsilon^{n-2}\int_{\partial \Omega}R_{q,n}(\epsilon(t-s))\theta^i(s)\, d\sigma_s\\
&\quad -\frac{1}{|\partial\Omega|_{n-1}}\int_{\partial \Omega}\Bigl(v^+[\partial \Omega,\theta^i](s')+\epsilon^{n-2}\int_{\partial \Omega}R_{q,n}(\epsilon(s'-s))\theta^i(s)d\sigma_s\,\Bigr) d\sigma_{s'}\\
&\quad -v^-[\partial \Omega,\theta^o](t)-\epsilon^{n-2}\int_{\partial \Omega}R_{q,n}(\epsilon(t-s))\theta^o(s)\, d\sigma_s\\
&\quad +\frac{1}{|\partial\Omega|_{n-1}}\int_{\partial \Omega}\Bigl(v^-[\partial \Omega,\theta^o](s')+\epsilon^{n-2}\int_{\partial \Omega}R_{q,n}(\epsilon(s'-s))\theta^o(s)d\sigma_s\,\Bigr) d\sigma_{s'}\Biggr)\\
&-g(t)+\frac{1}{|\partial\Omega|_{n-1}}\int_{\partial \Omega}g\,d\sigma+\lambda^+(\nu_{\Omega}(t))_j \qquad \forall t \in \partial \Omega\, ,
\end{split}
\end{equation}
for all $(\epsilon,\epsilon', \theta^i,\theta^o)\in ]-\epsilon_{0},\epsilon_{0}[\times \mathbb{R}\times
(C^{0,\alpha}(\partial\Omega)_{0})^2$. Then we have the following proposition.

\begin{proposition}\label{prop:finteq}
Let $\alpha\in]0,1[$. Let $p\in Q$. Let $\Omega$ be as in \eqref{dom}. Let $\epsilon_{0}$ be as in \eqref{e0}.  Let $\lambda^+, \lambda^- \in ]0,+\infty[$. Let $f \in C^{0,\alpha}(\partial \Omega)_0$. Let $g \in C^{0,\alpha}(\partial \Omega)$. Let $\rho$ be a function from $]0,\epsilon_0[$ to $]0,+\infty[$. Let $\epsilon \in ]0,\epsilon_0[$. Let $j \in \{1,\dots,n\}$. Then the unique solution $(u^+_j[\epsilon],u^-_j[\epsilon])$ in $C^{1,\alpha}_{\mathrm{loc}}(\mathrm{cl}\mathbb{S}[\Omega_{p,\epsilon}])\times C^{1,\alpha}_{\mathrm{loc}}(\mathrm{cl}\mathbb{S}[\Omega_{p,\epsilon}]^{-})$ of problem \eqref{bvpe} is delivered by
\[
\begin{split}
&u^+_j[\epsilon](x)\equiv v_q^+[\partial \Omega_{p,\epsilon},\hat{\theta}_j^i[\epsilon]((\cdot-p)/\epsilon)](x)-\frac{\epsilon^{1-n}}{|\partial \Omega|_{n-1}}\int_{\partial \Omega_{p,\epsilon}}\!\!\!v_q^+[\partial \Omega_{p,\epsilon},\hat{\theta}_j^i[\epsilon]((\cdot-p)/\epsilon)]\, d\sigma\\
&\qquad\qquad +x_j-\frac{\epsilon^{1-n}}{|\partial \Omega|_{n-1}}\int_{\partial \Omega_{p,\epsilon}}y_j\,d\sigma_y \qquad \forall x \in \mathrm{cl}\mathbb{S}[\Omega_{p,\epsilon}]\,,\\
& u^-_j[\epsilon](x)\equiv v_q^-[\partial \Omega_{p,\epsilon},\hat{\theta}_j^o[\epsilon]((\cdot-p)/\epsilon)](x)-\frac{\epsilon^{1-n}}{|\partial \Omega|_{n-1}}\int_{\partial \Omega_{p,\epsilon}}\!\!\!v_q^-[\partial \Omega_{p,\epsilon},\hat{\theta}_j^o[\epsilon]((\cdot-p)/\epsilon)]\, d\sigma\\
& -\rho(\epsilon)\frac{\epsilon^{1-n}}{|\partial \Omega|_{n-1}}\int_{\partial \Omega_{p,\epsilon}}g((y-p)/\epsilon)\, d\sigma_y +x_j-\frac{\epsilon^{1-n}}{|\partial \Omega|_{n-1}}\int_{\partial \Omega_{p,\epsilon}}\!\!\!y_j\,d\sigma_y \quad \forall x \in \mathrm{cl}\mathbb{S}[\Omega_{p,\epsilon}]^{-}\,,
\end{split}
\]
where $(\hat{\theta}_j^i[\epsilon],\hat{\theta}_j^o[\epsilon])$ denotes the unique solution $(\theta^i,\theta^o)$ in $(C^{0,\alpha}(\partial \Omega)_0)^2$ of
\begin{equation}
\label{lambdab}
M\Bigl[\epsilon,\frac{\epsilon}{\rho(\epsilon)},\theta^i,\theta^o\Bigr]=0\,.
\end{equation}
\end{proposition}
\begin{proof}
We first note that the unique solution $(u^+_j[\epsilon],u^-_j[\epsilon])$ in $C^{1,\alpha}_{\mathrm{loc}}(\mathrm{cl}\mathbb{S}[\Omega_{p,\epsilon}])\times C^{1,\alpha}_{\mathrm{loc}}(\mathrm{cl}\mathbb{S}[\Omega_{p,\epsilon}]^{-})$ of problem \eqref{bvpe} is delivered by
\begin{align}
&u^+_j[\epsilon](x)\equiv v^+_j[\epsilon](x)+x_j\qquad \forall x \in \mathrm{cl}\mathbb{S}[\Omega_{p,\epsilon}]\, ,\nonumber\\
& u^-_j[\epsilon](x)\equiv v^-_j[\epsilon](x)+x_j\qquad \forall x \in \mathrm{cl}\mathbb{S}[\Omega_{p,\epsilon}]^-\,,  \nonumber
\end{align}
where $(v^+_j[\epsilon],v^-_j[\epsilon])$ is the unique solution in $C^{1,\alpha}_{q}(\mathrm{cl}\mathbb{S}[\Omega_{p,\epsilon}])\times C^{1,\alpha}_{q}(\mathrm{cl}\mathbb{S}[\Omega_{p,\epsilon}]^{-})$ of the following periodic problem
\begin{equation}
\label{inteq3}
\left\{
\begin{array}{ll}
\Delta v^+=0 & {\mathrm{in}}\ {\mathbb{S}}[\Omega_{p,\epsilon}]\,,\\
\Delta v^-=0 & {\mathrm{in}}\ {\mathbb{S}}[\Omega_{p,\epsilon}]^{-}\,,
\\
v^+(x+q_{hh}e_h)=v^+(x)& \forall x \in \mathrm{cl}\mathbb{S}[\Omega_{p,\epsilon}]\, ,\ \forall h \in\{1,\dots,n\}\,,\\
v^-(x+q_{hh}e_h)=v^-(x) & \forall x \in \mathrm{cl}\mathbb{S}[\Omega_{p,\epsilon}]^{-}\, ,\ \forall h \in\{1,\dots,n\}\,,\\
\lambda^-\frac{\partial v^-}{\partial\nu_{ \Omega_{p,\epsilon} }}(x)-\lambda^+\frac{\partial v^+}{\partial\nu_{ \Omega_{p,\epsilon}}}(x)\\
 =f((x-p)/\epsilon)+(\lambda^+-\lambda^-)(\nu_{\Omega_{p,\epsilon}}(x))_j& \forall x\in  \partial\Omega_{p,\epsilon}\,,\\
\lambda^+\frac{\partial v^+}{\partial\nu_{\Omega_{p,\epsilon}}}(x)+\frac{1}{\rho(\epsilon)}\bigl(v^+(x)-v^-(x)\bigr)\\
 =g((x-p)/\epsilon)-\lambda^+(\nu_{\Omega_{p,\epsilon}}(x))_j& \forall x\in  \partial\Omega_{p,\epsilon}\,,\\
\int_{\partial \Omega_{p,\epsilon}}v^+(x)\, d\sigma_x=-\int_{\partial \Omega_{p,\epsilon}}y_j\,d\sigma_y\, .
\end{array}
\right.
\end{equation}
Then by Proposition \ref{prop:J}, by Theorem \ref{thm:ex}, and by a simple computation based on the rule of change of variables in integrals, we deduce the validity of the proposition.
\qquad\end{proof}

By Proposition \ref{prop:finteq}, we are reduced to analyse the system of integral equations \eqref{lambdab}. We note that equation \eqref{lambdab} does not make sense for $\epsilon=0$, but we can consider the equation
\begin{equation}
\label{lambdac}
M[\epsilon,\epsilon',\theta^i,\theta^o]=0\,,
\end{equation}
around $(\epsilon, \epsilon')=(0,r_{\ast})$, and equation \eqref{lambdac} makes perfectly sense if 
\begin{equation}
\label{lambdad}
\epsilon=0\,,\quad\epsilon'=r_{\ast} \, .
\end{equation}
 By Proposition  \ref{prop:finteq}, we already know that if $\epsilon\in ]0,\epsilon_{0}[$ and  $\epsilon'=\epsilon/\rho(\epsilon)$, then equation \eqref{lambdac} is equivalent to problem \eqref{bvpe}. We now consider equation \eqref{lambdac} under condition \eqref{lambdad}  by means of the following.

\begin{theorem}
\label{thm:lim}
Let  $\alpha\in]0,1[$. Let $p\in Q$. Let $\Omega$ be as in \eqref{dom}. Let $\epsilon_{0}$ be as in \eqref{e0}.  Let $\lambda^+, \lambda^- \in ]0,+\infty[$. Let $f \in C^{0,\alpha}(\partial \Omega)_0$. Let $g \in C^{0,\alpha}(\partial \Omega)$. Let $\rho$ be a function from $]0,\epsilon_0[$ to $]0,+\infty[$. Let $j \in \{1,\dots,n\}$. Let assumption \eqref{varrhorast} hold. Let  $r_\ast$ be as in \eqref{gast}. Let $M$ be as in \eqref{Lmbdeq1}-\eqref{Lmbdeq2}. Then there exists a unique pair $(\theta^i,\theta^o) \in (C^{0,\alpha}(\partial \Omega)_0)^2$ such that
\begin{equation}\label{eq:lim0}
M[0,r_{\ast},\theta^i,\theta^o]=0\, ,
\end{equation}
and we denote such a pair by $(\tilde{\theta}^i_j,\tilde{\theta}^o_j)$. Moreover, the pair of functions $(\tilde{u}^+_j,\tilde{u}^-_j) \in C^{1,\alpha}(\mathrm{cl}\Omega)\times C^{1,\alpha}_{\mathrm{loc}}( \mathbb{R}^n \setminus  \Omega)$, defined by
\begin{align}
&\tilde{u}^+_j\equiv v^+[\partial\Omega,\tilde{\theta}^i_j]-\frac{1}{|\partial\Omega|_{n-1}}\int_{\partial \Omega}v^+[\partial\Omega,\tilde{\theta}^i_j]\,d\sigma\,,\label{eq:lim1}\\
& \tilde{u}^-_j\equiv v^-[\partial\Omega,\tilde{\theta}^o_j]-\frac{1}{|\partial\Omega|_{n-1}}\int_{\partial \Omega}v^-[\partial\Omega,\tilde{\theta}^o_j]\,d\sigma\, , \label{eq:lim2}
\end{align}
is the unique solution of the following  \emph{`limiting boundary value problem'}
\begin{equation}\label{eq:lim3}
\left\{
\begin{array}{ll}
\Delta u^+=0 & {\mathrm{in}}\ \Omega\,,\\
\Delta u^-=0 & {\mathrm{in}}\ \mathbb{R}^n \setminus \mathrm{cl}\Omega\,,
\\
\lambda^-\frac{\partial u^-}{\partial\nu_{ \Omega}}(x)-\lambda^+\frac{\partial u^+}{\partial\nu_{ \Omega}}(x)=f(x)+(\lambda^+-\lambda^-)(\nu_\Omega(x))_j& \forall x\in  \partial\Omega\,,\\
\lambda^+\frac{\partial u^+}{\partial\nu_{ \Omega}}(x)+r_{\ast}\Bigl(u^+(x)-u^-(x)\Bigr)=g(x)-\frac{\int_{\partial \Omega}g\, d\sigma}{|\partial\Omega|_{n-1}}-\lambda^+(\nu_{\Omega}(x))_j& \forall x\in  \partial\Omega\,,\\
\int_{\partial \Omega}u^+(x)\, d\sigma_x=0\, ,\\
\int_{\partial \Omega}u^-(x)\, d\sigma_x=0\, ,\\
\lim_{x\to \infty}u^-(x)\in\mathbb{R}\, .
\end{array}
\right.
\end{equation}
\end{theorem}
\begin{proof}
We first note that equation \eqref{eq:lim0} can be rewritten as
\[
K_{r_{\ast}}[\theta^i,\theta^o]=\left(f+(\lambda^+-\lambda^-)(\nu_\Omega)_j \,,\ g-\frac{1}{|\partial\Omega|_{n-1}}\int_{\partial \Omega}g\, d\sigma-\lambda^+(\nu_{\Omega})_j\right)
\]
(see Proposition \ref{prop:K}). Then, by Proposition \ref{prop:K}, there exists a unique pair $(\theta^i,\theta^o) \in (C^{0,\alpha}(\partial \Omega)_0)^2$ such that \eqref{eq:lim0} holds. Then by Theorem \ref{thm:exbis} and by classical potential theory, the pair of functions delivered by \eqref{eq:lim1}-\eqref{eq:lim2} is the unique solution of problem \eqref{eq:lim3}.\qquad\end{proof}

Let $\tilde{\theta}^o_j$, $\tilde{u}^-_j$ be as in Theorem \ref{thm:lim}. Then by classical potential theory and by equality $\int_{\partial \Omega}\tilde{\theta}^o_j\, d\sigma=0$, we observe that we have
\begin{equation}\label{rem:lim1}
\tilde{l}^-_j \equiv \lim_{x\to\infty}\tilde{u}^-_j(x)=-\frac{1}{|\partial\Omega|_{n-1}}\int_{\partial \Omega}v^-[\partial\Omega,\tilde{\theta}^o_j]\,d\sigma\,. 
\end{equation}

We now turn to analyse equation \eqref{lambdac} for $(\epsilon,\epsilon')$ in a neighbourhood of $(0,r_\ast)$ by means of the following.
\begin{theorem}\label{thm:Lmbd}
Let $\alpha\in]0,1[$. Let $p\in Q$. Let $\Omega$ be as in \eqref{dom}. Let $\epsilon_{0}$ be as in \eqref{e0}.  Let $\lambda^+, \lambda^- \in ]0,+\infty[$. Let $f \in C^{0,\alpha}(\partial \Omega)_0$. Let $g \in C^{0,\alpha}(\partial \Omega)$. Let $\rho$ be a function from $]0,\epsilon_0[$ to $]0,+\infty[$. Let $j \in \{1,\dots,n\}$. Let assumption \eqref{varrhorast} hold. Let $r_\ast$ be as in \eqref{gast}. Let $M$ be as in \eqref{Lmbdeq1}-\eqref{Lmbdeq2}. Let $(\hat{\theta}^i_j[\cdot],\hat{\theta}^o_j[\cdot])$ be as in Proposition \ref{prop:finteq}. Let $(\tilde{\theta}^i_j,\tilde{\theta}^o_j)$ be as in Theorem \ref{thm:lim}. Then there exist  $\epsilon_1\in]0,\epsilon_{0}]$,   an open neighbourhood ${\mathcal{U}}_{r_{\ast}}$ of  $r_{\ast}$ in ${\mathbb{R}}$, an open neighbourhood ${\mathcal{V}}$ of $(\tilde{\theta}^i_j,\tilde{\theta}^o_j)$ in $(C^{0,\alpha}(\partial\Omega)_{0})^2$, and a real analytic operator $(\Theta^i_j,\Theta^o_j)$ from 
$]-\epsilon_1,\epsilon_1[ \times {\mathcal{U}}_{r_{\ast}}$ to ${\mathcal{V}}$ such that $\epsilon/\rho(\epsilon)\in  \mathcal{U}_{r_{\ast}}$ for all $\epsilon\in]0,\epsilon_1[$, and such that the set of zeros of $M$ in $]-\epsilon_1,\epsilon_1[ \times {\mathcal{U}}_{r_{\ast}}\times {\mathcal{V}}$ coincides with the graph of $(\Theta^i_j,\Theta^o_j)$. In particular, 
\[
\Bigl(\Theta^i_j\Bigl[\epsilon, \frac{\epsilon}{\rho(\epsilon)}\Bigr],\Theta^o_j\Bigl[\epsilon, \frac{\epsilon}{\rho(\epsilon)}\Bigr]\Bigl)=\Bigl(\hat{\theta}^i_j[\epsilon],\hat{\theta}^o_j[\epsilon]\Bigr)\  \forall \epsilon \in ]0,\epsilon_1[\,,
\ 
(\Theta^i_j[0 ,r_{\ast}],\Theta^o_j[0, r_{\ast}])=(\tilde{\theta}^i_j,\tilde{\theta}^o_j)\,.
\]
\end{theorem}
 \begin{proof} We plan to apply the Implicit Function Theorem to equation \eqref{lambdac} around the point $(0,  r_{\ast},\tilde{\theta}^i_j,\tilde{\theta}^o_j)$. By standard properties of integral  operators with real analytic kernels and with no singularity, and by classical mapping properties of layer potentials  (cf.~\cite[\S 4]{LaMu11a},  Miranda~\cite{Mi65}, Lanza de Cristoforis and Rossi \cite[Thm.~3.1]{LaRo04}), we conclude that $M$ is real analytic. By definition of $(\tilde{\theta}^i_j,\tilde{\theta}^o_j)$, we have $M[0, r_{\ast},\tilde{\theta}^i_j,\tilde{\theta}^o_j]=0$. By standard calculus in Banach spaces, the differential of $M$ at the point $(0,r_{\ast},\tilde{\theta}^i_j,\tilde{\theta}^o_j)$ with respect to the variables $(\theta^i,\theta^o)$ is delivered by the formula
\[
\partial_{(\theta^i,\theta^o)}M[0,r_{\ast},\tilde{\theta}^i_j,\tilde{\theta}^o_j]
(\overline{\theta}^i,\overline{\theta}^o)
=K_{r_{\ast}}[\overline{\theta}^i,\overline{\theta}^o]\qquad\qquad\forall (\overline{\theta}^i,\overline{\theta}^o)\in 
(C^{0,\alpha}(\partial\Omega)_{0})^2
\]
 (see Proposition \ref{prop:K}). 
 Then by Proposition \ref{prop:K}, $\partial_{(\theta^i,\theta^o)}M[0, r_{\ast},\tilde{\theta}^i_j,\tilde{\theta}^o_j]$ is a linear homeomorphism from $(C^{0,\alpha}(\partial\Omega)_{0})^2$ onto $(C^{0,\alpha}(\partial\Omega)_{0})^2$. Hence the existence of $\epsilon_1$, ${\mathcal{U}}_{r_{\ast}}$, ${\mathcal{V}}$, $\Theta^i_j$, $\Theta^o_j$ as in the statement follows by the Implicit Function Theorem for real analytic maps in Banach spaces (cf., \textit{e.g.},  Deimling \cite[Theorem 15.3]{De85}). \qquad\end{proof}

\section{A functional analytic representation theorem for the solutions of problem \eqref{bvpe}}
\label{fure}

In the following Theorem \ref{thm:rep+} we investigate the behaviour of  $u^+_j[\epsilon]$ for $\epsilon$ small and positive.

\begin{theorem}\label{thm:rep+}
Let the assumptions of Theorem \ref{thm:Lmbd} hold. 
Then there exists a real analytic map $U^{+}_j$ from $]-\epsilon_1,\epsilon_1[\times{\mathcal{U}}_{r_{\ast}}$ to $C^{1,\alpha}({\mathrm{cl}}\Omega)$ such that
\[
u^+_j[\epsilon](p+\epsilon t)=\epsilon
U^{+}_j\Bigl[\epsilon, \frac{\epsilon}{\rho(\epsilon)}\Bigr](t)\qquad\forall t\in  {\mathrm{cl}}\Omega\,,
\]
for all $\epsilon\in]0,\epsilon_1[$, where $(u^+_j[\epsilon], u^-_j[\epsilon])$ is the unique solution of problem \eqref{bvpe}. Moreover, 
\begin{equation}
\label{eq:rep2}
U^{+}_j[0, r_{\ast}](t)=\tilde{u}^+_j(t)+t_j-\frac{1}{|\partial\Omega|_{n-1}}\int_{\partial\Omega}s_j\, d\sigma_s\qquad\forall t\in {\mathrm{cl}}\Omega\,, 
\end{equation} where $\tilde{u}_j^+$ is defined as in Theorem \ref{thm:lim}.
\end{theorem}
 \begin{proof} If $\epsilon \in ]0,\epsilon_1[$, then a simple computation based on the rule of change of variables in integrals shows that
\[
\begin{split}
u^+_j[\epsilon](p+\epsilon t)= &\epsilon v^+\Bigl[\partial \Omega,\Theta^i_j\Bigl[\epsilon, \frac{\epsilon}{\rho(\epsilon)}\Bigr]\Bigr](t) +\epsilon^{n-1}\int_{\partial \Omega} R_{q,n}(\epsilon (t-s))\Theta^i_j\Bigl[\epsilon, \frac{\epsilon}{\rho(\epsilon)}\Bigr](s)\, d\sigma_s\\
&-\frac{\epsilon}{|\partial\Omega|_{n-1}}\int_{\partial \Omega}\Bigl(v^+\Bigl[\partial \Omega,\Theta^i_j\Bigl[\epsilon, \frac{\epsilon}{\rho(\epsilon)}\Bigr]\Bigr](s')\\
&\qquad\qquad\qquad \quad+\epsilon^{n-2}\int_{\partial \Omega}R_{q,n}(\epsilon(s'-s))\Theta^i_j\Bigl[\epsilon, \frac{\epsilon}{\rho(\epsilon)}\Bigr](s)d\sigma_s\Bigr) d\sigma_{s'}\\
& +\epsilon t_j-\frac{\epsilon}{|\partial\Omega|_{n-1}}\int_{\partial \Omega}s_j\,d\sigma_s \qquad \forall t \in \mathrm{cl}\Omega\,
\end{split}
\]
(see also Proposition \ref{prop:finteq} and Theorem \ref{thm:Lmbd}). Therefore it is natural to set
\[
\begin{split}
&U_j^{+}[\epsilon,\epsilon'](t)\equiv v^+\Bigl[\partial \Omega,\Theta^i_j\bigl[\epsilon, \epsilon'\bigr]\Bigr](t) +\epsilon^{n-2}\int_{\partial \Omega} R_{q,n}(\epsilon (t-s))\Theta^i_j\bigl[\epsilon, \epsilon'\bigr](s)\, d\sigma_s\\
&-\frac{1}{|\partial\Omega|_{n-1}}\int_{\partial \Omega}\Bigl(v^+\Bigl[\partial \Omega,\Theta^i_j\bigl[\epsilon, \epsilon'\bigr]\Bigr](s')+\epsilon^{n-2}\int_{\partial \Omega}R_{q,n}(\epsilon(s'-s))\Theta^i_j\bigl[\epsilon, \epsilon'\bigr](s)d\sigma_s\,\Bigr) d\sigma_{s'}\\
& +t_j-\frac{1}{|\partial\Omega|_{n-1}}\int_{\partial\Omega}s_j\,d\sigma_s  \qquad\qquad \qquad \qquad \forall t \in \mathrm{cl}\Omega\, ,
\end{split}
\]
 for all $(\epsilon,\epsilon') \in ]-\epsilon_1,\epsilon_1[\times{\mathcal{U}}_{r_{\ast}}$. By standard properties of integral  operators with real analytic kernels and with no singularity, by classical mapping properties of layer potentials  (cf.~\cite[\S4]{LaMu11a},  Miranda~\cite{Mi65}, Lanza de Cristoforis and Rossi \cite[Thm.~3.1]{LaRo04}) and by Theorem \ref{thm:Lmbd},  we conclude that $U_j^{+}$ is real analytic. Moreover, Theorem \ref{thm:Lmbd} implies that $\Theta^i_j[0,r_{\ast}]=\tilde{\theta}^i_j$ and thus
 the validity of equality \eqref{eq:rep2} follows (see also Theorem \ref{thm:lim}). \qquad\end{proof}

In the following Theorem \ref{thm:rep-} we investigate the behaviour of  $u^-_j[\epsilon]$ for $\epsilon$ small and positive.

\begin{theorem}\label{thm:rep-}
Let the assumptions of Theorem \ref{thm:Lmbd} hold. Let $(u^+_j[\epsilon], u^-_j[\epsilon])$ be the unique solution of problem \eqref{bvpe} for all $\epsilon \in ]0,\epsilon_0[$. Let $\tilde{l}^-_j$ be as in \eqref{rem:lim1}. Then there exists a real analytic operator $C^-_j$   from $]-\epsilon_1,\epsilon_1[\times{\mathcal{U}}_{r_{\ast}}$ to ${\mathbb{R}}$  such that 
\begin{equation}
\label{eq:rep3}
C^{-}_{j}[0,r_{\ast}]=-\frac{1}{|\partial\Omega|_{n-1}}\int_{\partial\Omega}g\,d\sigma+r_\ast \tilde{l}_j^--\frac{r_\ast}{|\partial\Omega|_{n-1}}\int_{\partial\Omega}s_j\,d\sigma_s\, ,
\end{equation}
and such that the following statements hold.
\begin{enumerate}
\item[(i)]    Let $\tilde{\Omega}$ be an open bounded subset of ${\mathbb{R}}^{n}$ such that  $\mathrm{cl}\tilde{\Omega}\cap(p+q{\mathbb{Z}}^{n})=\emptyset$. Let $k \in \mathbb{N}$.
 Then  there exist
$\epsilon_{\tilde{\Omega}}\in]0,\epsilon_1[$ 
and a real analytic operator $U^{-}_{j,\tilde\Omega}$ from $]-\epsilon_{ \tilde{\Omega} },\epsilon_{ \tilde{\Omega} }[\times{\mathcal{U}}_{r_{\ast}}$ to $C^{k}(
{\mathrm{cl}} \tilde{\Omega})$  such that $\mathrm{cl}\tilde{\Omega}\subseteq {\mathbb{S}}[ \Omega_{p,\epsilon}]^{-}$ for all $\epsilon\in ]-\epsilon_{\tilde{\Omega}},\epsilon_{\tilde{\Omega}}[$, and such that
\begin{equation}\label{eq:rep5}
u^-_j[\epsilon](x)=x_j-p_j+\rho(\epsilon)C_j^-\Bigl[\epsilon, \frac{\epsilon}{\rho(\epsilon)}\Bigr]+ \epsilon^n U^{-}_{j,\tilde{\Omega}} \Bigl[\epsilon, \frac{\epsilon}{\rho(\epsilon)}\Bigr](x)\quad\forall x\in {\mathrm{cl}}\tilde{\Omega}
\end{equation}
 for all $\epsilon\in]0,\epsilon_{\tilde{\Omega}}[$. Moreover, 
\begin{equation}
\label{eq:rep4}
\begin{split}
&U^{-}_{j,\tilde{\Omega}}[0,r_{\ast}](x)\\
& =DS_{q,n}(x-p)\left(\int_{\partial\Omega}  \nu_\Omega(s)
\tilde{u}^-_j(s)\, d\sigma_s-\int_{\partial\Omega}  s\frac{\partial}{\partial \nu_{\Omega}} \tilde{u}^-_j(s) \, d\sigma_s
\right)\quad \forall x\in {\mathrm{cl}}\tilde{\Omega}
\end{split}
\end{equation}
where $\tilde{u}^-_j$ is defined as in Theorem \ref{thm:lim}.

\item[(ii)] Let $\tilde{\Omega}$ be a bounded open subset of ${\mathbb{R}}^{n}\setminus{\mathrm{cl}}\Omega$. Then  there exist  $\epsilon^\#_{\tilde{\Omega}}\in ]0,\epsilon_1[$ and a real analytic map $V^{-}_{ j, \tilde{\Omega} }$ from $]-\epsilon^\#_{\tilde{\Omega}},\epsilon^\#_{\tilde{\Omega}}[\times{\mathcal{U}}_{r_{\ast}}$ to $C^{1,\alpha}({\mathrm{cl}}\tilde{\Omega})$ such that $p+\epsilon {\mathrm{cl}}\tilde{\Omega}\subseteq {\mathrm{cl}}{\mathbb{S}} [\Omega_{p,\epsilon}]^{-}$  for all $\epsilon\in
]-\epsilon^\#_{\tilde{\Omega}},\epsilon^\#_{\tilde{\Omega}}[$, and
\[
u^-_j[\epsilon](p+\epsilon t)=\rho(\epsilon)C_j^-\Bigl[\epsilon, \frac{\epsilon}{\rho(\epsilon)}\Bigr]+\epsilon
V^{-}_{ j,\tilde{\Omega} }\Bigl[\epsilon, \frac{\epsilon}{\rho(\epsilon)}\Bigr](t)\qquad\forall t\in {\mathrm{cl}}\tilde{\Omega} 
\]
 for all $\epsilon\in]0,\epsilon^\#_{\tilde{\Omega}}[$. Moreover, 
\begin{equation}
\label{eq:rep7}
V^{-}_{j, \tilde{\Omega} }[0,r_{\ast}](t)=\tilde{u}^-_j(t)-\tilde{l}^-_j+t_j\qquad\forall t\in {\mathrm{cl}}\tilde{\Omega} 
\end{equation} where $\tilde{u}^-_j$ is defined as in Theorem \ref{thm:lim}.
\end{enumerate}
\end{theorem}
 \begin{proof}We set 
 \begin{equation}\label{cj-}
 \begin{split}
 &C_j^-[\epsilon,\epsilon']\equiv -\frac{1}{ |\partial\Omega|_{n-1}}\int_{\partial \Omega}g\, d\sigma-\frac{\epsilon'}{|\partial\Omega|_{n-1}}\int_{\partial \Omega}\Bigl(v^-\Bigl[\partial \Omega,\Theta^o_j\bigl[\epsilon, \epsilon'\bigr]\Bigr](t)\\&\ +\epsilon^{n-2}\int_{\partial \Omega}R_{q,n}(\epsilon(t-s))\Theta^o_j\bigl[\epsilon, \epsilon'\bigr](s)d\sigma_s\,\Bigr) d\sigma_t-\frac{\epsilon'}{|\partial\Omega|_{n-1}}\int_{\partial \Omega}t_j\, d\sigma_t
 \end{split}
 \end{equation}
 for all $(\epsilon,\epsilon') \in  ]-\epsilon_1 ,\epsilon_1[\times{\mathcal{U}}_{r_{\ast}}$.   By standard properties of integral  operators with real analytic kernels and with no singularity, by classical mapping properties of layer potentials  (cf.~\cite[\S   4]{LaMu11a},  Miranda~\cite{Mi65}, Lanza de Cristoforis and Rossi \cite[Thm.~3.1]{LaRo04}) and by Theorem \ref{thm:Lmbd},  we deduce that $C_j^-$ is real analytic.     Then,  by \eqref{rem:lim1} and by equality $\Theta^o_j[0,r_\ast]=\tilde{\theta}^o_j$ one verifies the validity of \eqref{eq:rep3}.

We now  consider the proof of statement (i). By taking $\epsilon_{\tilde{\Omega}}$ small enough, we can  assume that $\mathrm{cl}\tilde{\Omega}\subseteq {\mathbb{S}}[ \Omega_{p,\epsilon}]^{-}$  for all $\epsilon\in [-\epsilon_{\tilde{\Omega}},\epsilon_{\tilde{\Omega}}]$. By Proposition \ref{prop:finteq} and Theorem \ref{thm:Lmbd}, we have 
\[
\begin{split}
 u^-_j[\epsilon](x)=&\epsilon^{n-1}\int_{\partial \Omega} S_{q,n}(x-p-\epsilon s)\Theta^o_j\Bigl[\epsilon, \frac{\epsilon}{\rho(\epsilon)}\Bigr](s)\, d\sigma_s\\&-\frac{\epsilon}{|\partial\Omega|_{n-1}}\int_{\partial \Omega}\Bigl(v^-\Bigl[\partial \Omega,\Theta^o_j\Bigl[\epsilon, \frac{\epsilon}{\rho(\epsilon)}\Bigr]\Bigr](t)\\
 &\qquad \qquad \qquad \quad+\epsilon^{n-2}\int_{\partial \Omega}R_{q,n}(\epsilon(t-s))\Theta^o_j\Bigl[\epsilon, \frac{\epsilon}{\rho(\epsilon)}\Bigr](s)d\sigma_s\,\Bigr) d\sigma_t\\
&-\frac{\rho(\epsilon)}{|\partial\Omega|_{n-1}}\int_{\partial \Omega}g\, d\sigma +x_j -p_j-\frac{\epsilon}{|\partial\Omega|_{n-1}}\int_{\partial \Omega}s_j\,d\sigma_s \qquad\qquad \forall x \in \mathrm{cl}\tilde{\Omega}\,,
\end{split}
\]
for all $\epsilon\in]0,\epsilon_{\tilde{\Omega}}[$. Then we note that if $\epsilon\in [-\epsilon_{ \tilde{\Omega}},\epsilon_{ \tilde{\Omega}}] $, then 
$\partial{\mathbb{S}}[\Omega_{p,\epsilon}]\cap {\mathrm{cl}}{\tilde{\Omega}}=\emptyset$.
Also, we observe that if $x\in {\mathrm{cl}}\tilde{\Omega}$, then   $x-p-\epsilon \beta s$ does not belong to $q {\mathbb{Z}}^{n}$
for any $s\in\partial\Omega$, $\epsilon \in]-\epsilon_{ \tilde{\Omega}},\epsilon_{ \tilde{\Omega}}[$, and $\beta\in[0,1]$. Accordingly, we can invoke the Taylor formula with integral residue and write
\[
S_{q,n}(x-p-\epsilon s)-S_{q,n}(x-p)=
-\epsilon\int_{0}^{1}DS_{q,n}(x-p-\beta\epsilon s)  s\,d\beta
  \,,
\]
for all $(x,s)\in{\mathrm{cl}} \tilde{\Omega} \times\partial\Omega$ and $\epsilon\in]0,\epsilon_{\tilde{\Omega}}[$. Since
$\int_{\partial\Omega}\Theta^o_j\bigl[\epsilon, \epsilon/\rho(\epsilon)\bigr]\,d\sigma=0$ for all $\epsilon\in]0,\epsilon_{\tilde{\Omega}}[$, we conclude that
\begin{equation}\label{eq:rep6}
\begin{split}
u^-_j[\epsilon](x)
=&- \epsilon^{n }
\int_{\partial\Omega}
\left(\int_{0}^{1}DS_{q,n}(x-p-\beta\epsilon s)  s\,d\beta\right)
\Theta^o_j\Bigl[\epsilon, \frac{\epsilon}{\rho(\epsilon)}\Bigr](s)\,d\sigma_{s}\\
&-\frac{\epsilon}{|\partial\Omega|_{n-1}}\int_{\partial \Omega}\biggl(v^-\Bigl[\partial \Omega,\Theta^o_j\Bigl[\epsilon, \frac{\epsilon}{\rho(\epsilon)}\Bigr]\Bigr](t)\\
&\qquad \qquad \qquad \quad+\epsilon^{n-2}\int_{\partial \Omega}R_{q,n}(\epsilon(t-s))\Theta^o_j\Bigl[\epsilon, \frac{\epsilon}{\rho(\epsilon)}\Bigr](s)d\sigma_s\,\biggr) d\sigma_t\\
&-\frac{\rho(\epsilon)}{|\partial\Omega|_{n-1}}\int_{\partial \Omega}g\, d\sigma +x_j -p_j-\frac{\epsilon}{|\partial\Omega|_{n-1}}\int_{\partial \Omega}s_j\,d\sigma_s
\end{split}
\end{equation} for all $x\in
{\mathrm{cl}}\tilde{\Omega}$ and 
for all $\epsilon\in]0,\epsilon_{\tilde{\Omega}}[$. Thus it is natural to set
\begin{equation}\label{eq:rep7a}
\begin{split}
&U^{-}_{j, \tilde{\Omega} }[\epsilon,\epsilon'](x)\\
&\quad\equiv - \int_{\partial\Omega}\left(\int_{0}^{1}DS_{q,n}(x-p-\beta\epsilon s)  s\,d\beta\right)\Theta^o_j\bigl[\epsilon, \epsilon' \bigr](s)\,d\sigma_{s}
\quad\forall x\in \mathrm{cl}\tilde{\Omega}\,
\end{split}
\end{equation}
for all $(\epsilon,\epsilon') \in  ]-\epsilon_{\tilde{\Omega}} ,\epsilon_{\tilde{\Omega}}[\times{\mathcal{U}}_{r_{\ast}}$.  Then the validity of \eqref{eq:rep5} follows by definitions \eqref{cj-}, \eqref{eq:rep7a} and equality \eqref{eq:rep6}. 
By standard properties of integral  operators with real analytic kernels and with no singularity (cf., \textit{e.g.},  \cite[\S 3]{LaMu11a}) and by arguing exactly as in the proof of statement (i) of \cite[Thm.~5.1]{LaMu12}, one verifies that $U^{-}_{j,\tilde{\Omega}}$ defines a real analytic map from $]-\epsilon_{\tilde{\Omega}} ,\epsilon_{\tilde{\Omega}}[\times{\mathcal{U}}_{r_{\ast}}$ to $C^{k}({\mathrm{cl}} \tilde{\Omega})$. Next we turn to prove formula \eqref{eq:rep4}. 
Theorem \ref{thm:Lmbd} implies that $\Theta^o_j[0,r_{\ast}]=\tilde{\theta}^o_j$. Then we fix
 $k \in \{1,\dots,n\}$. By well known jump formulae for the normal derivative of the classical simple layer potential, we have
\[
\int_{\partial \Omega}s_k\tilde{\theta}^o_j(s)\, d\sigma_s=\int_{\partial \Omega}s_k \frac{\partial }{\partial \nu_\Omega}v^-[\partial \Omega,\tilde{\theta}^o_j](s)\,d\sigma_s-\int_{\partial \Omega}s_k \frac{\partial }{\partial \nu_\Omega}v^+[\partial \Omega,\tilde{\theta}^o_j](s)\,d\sigma_s\, .
\]
Then by the Green Identity, we have
\[
\int_{\partial \Omega}s_k \frac{\partial }{\partial \nu_\Omega}v^+[\partial \Omega,\tilde{\theta}^o_j](s)\,d\sigma_s=\int_{\partial \Omega} (\nu_\Omega(s))_k v^+[\partial \Omega,\tilde{\theta}^o_j](s)\,d\sigma_s\, .
\]
Moreover,
\[
 \frac{\partial }{\partial \nu_\Omega}\tilde{u}^-_j=\frac{\partial }{\partial \nu_\Omega}v^-[\partial \Omega,\tilde{\theta}^o_j] \qquad \textrm{on $\partial \Omega$}\, ,
\]
and
\begin{equation}\label{intnu}
\int_{\partial \Omega} (\nu_\Omega(s))_k \,d\sigma_s=0\, .
\end{equation}
As a consequence,
\[
\int_{\partial \Omega}s_k\tilde{\theta}^o_j(s)\, d\sigma_s=\int_{\partial \Omega}s_k \frac{\partial }{\partial \nu_\Omega}\tilde{u}^-_j(s)\, d\sigma_s-\int_{\partial \Omega} (\nu_\Omega(s))_k \tilde{u}^-_j(s)\, d\sigma_s\,,
\]
and accordingly \eqref{eq:rep4} holds.
 
We now consider statement (ii). By assumption, there exists $R>0$ such that $({\mathrm{cl}}\tilde{\Omega}\cup{\mathrm{cl}}{\Omega})\subseteq {\mathbb{B}}_{n}(0,R)$. Then we set $\Omega^\ast\equiv {\mathbb{B}}_{n}(0,R)\setminus{\mathrm{cl}}\Omega$.  Then there exists $\epsilon^\#_{\Omega^\ast}\in ]0,\epsilon_1[$ such that $p+\epsilon {\mathrm{cl}}\Omega^\ast \subseteq Q$, and $p+\epsilon\Omega^\ast\subseteq 
{\mathbb{S}}[\Omega_{p,\epsilon}]^{-}$, for all $\epsilon\in[-\epsilon^\#_{ \Omega^\ast} ,\epsilon^\#_{ \Omega^\ast} ]\setminus\{0\}$ (cf.~\cite[Lemma A.5 (ii)]{LaMu12}). Then we set $\epsilon^\#_{\tilde{\Omega}}\equiv \epsilon^\#_{\Omega^\ast}$. It clearly suffices to show that $V^{-}_{ j,\Omega^\ast }$ exists and then to set 
$V^{-}_{j, \tilde{\Omega} }$ equal to the composition of the restriction of $C^{1,\alpha}({\mathrm{cl}}\Omega^\ast)$ to $C^{1,\alpha}({\mathrm{cl}}\tilde{\Omega})$ with $V^{-}_{ j,\Omega^\ast }$. The advantage of $\Omega^\ast$ with respect to $\tilde{\Omega}$ is that $\Omega^\ast$  is of class $C^{1}$ and that accordingly $C^{2}({\mathrm{cl}}\Omega^\ast)$ is continuously imbedded into $C^{1,\alpha}({\mathrm{cl}}\Omega^\ast)$, a fact which we exploit below (cf., \textit{e.g.}, Lanza de Cristoforis \cite[Lem.~2.4 (ii)]{La91}). 

By equality
$\int_{\partial\Omega}\Theta^o_j\bigl[\epsilon, \epsilon/\rho(\epsilon)\bigr](s)\,d\sigma_{s}=0$,  
and by a simple computation based on the rule of change of variables in integrals, we have
\[
\begin{split}
&u^-_j[\epsilon](p+\epsilon t)
=\epsilon^{n-1}\int_{\partial\Omega}S_{q,n}(\epsilon (t-s))\Theta^o_j\Bigl[\epsilon, \frac{\epsilon}{\rho(\epsilon)}\Bigr](s)\,d\sigma_{s}+\rho(\epsilon)C^-_j\Bigl[\epsilon, \frac{\epsilon}{\rho(\epsilon)}\Bigr]+\epsilon t_j \\
&\ =\epsilon\biggl(\biggr.
\int_{\partial\Omega}S_{n}(  t-s)\Theta^o_j\Bigl[\epsilon, \frac{\epsilon}{\rho(\epsilon)}\Bigr](s)\,d\sigma_{s} +\epsilon^{n-2}  
\int_{\partial\Omega}R_{q,n}(\epsilon (t-s))\Theta^o_j\Bigl[\epsilon, \frac{\epsilon}{\rho(\epsilon)}\Bigr](s)\,d\sigma_{s}
\biggl.\biggr)\\
&
\quad +\rho(\epsilon)C^-_j\Bigl[\epsilon, \frac{\epsilon}{\rho(\epsilon)}\Bigr]+\epsilon t_j 
\qquad\qquad\qquad \forall t\in {\mathrm{cl}}\Omega^\ast\,,
\end{split}
\]
for all $\epsilon\in]0,\epsilon^\#_{\tilde{\Omega}}[$ (see also \eqref{cj-}). Thus it is natural to set
\begin{equation}
\label{est6}
\begin{split}
&V^{-}_{j,\Omega^\ast}[\epsilon,\epsilon'](t)
\equiv \int_{\partial\Omega}S_{n}(   t-s)\Theta^o_j [\epsilon,\epsilon'](s)\,d\sigma_{s}
\\ &\qquad
 \qquad
+ \epsilon^{n-2}  
\int_{\partial\Omega}R_{q,n}(\epsilon (t-s))\Theta^o_j [\epsilon,\epsilon'](s)\,d\sigma_{s}
+ t_j\qquad
\forall t\in {\mathrm{cl}}\Omega^\ast\,,
\end{split}
\end{equation}
for all $(\epsilon,\epsilon')\in ]-\epsilon^\#_{\tilde{\Omega}},\epsilon^\#_{\tilde{\Omega}}[\times \mathcal{U}_{r_{\ast}}$. Since $v^{-}[\partial\Omega,\cdot]_{|{\mathrm{cl}}
\Omega^\ast
}$ is linear and continuous from $C^{0,\alpha}(\partial\Omega)$ to $C^{1,\alpha}({\mathrm{cl}}\Omega^\ast)$ and $\Theta^o_j$ is real analytic, the map from $ ]-\epsilon^\#_{\tilde{\Omega}},\epsilon^\#_{\tilde{\Omega}}[\times \mathcal{U}_{r_{\ast}}$  to $C^{1,\alpha}({\mathrm{cl}}\Omega^\ast)$ which takes $(\epsilon,\epsilon')$ to the function $\int_{\partial\Omega}S_{n}(   t-s)\Theta^o_j [\epsilon,\epsilon'](s)\,d\sigma_{s}$ of the variable $t \in \mathrm{cl}\Omega^\ast$ is real analytic (cf., \textit{e.g.}, Miranda~\cite{Mi65}, Lanza  de Cristoforis and Rossi \cite[Thm.~3.1]{LaRo04}). 
Clearly, we have $(p+\epsilon {\mathrm{cl}}\Omega^\ast)\cap (\partial{\mathbb{S}}[\Omega_{p,\epsilon}]\setminus Q)=\emptyset$ for all $\epsilon\in ]-\epsilon^\#_{\tilde{\Omega}},\epsilon^\#_{\tilde{\Omega}}[$.
As a consequence, standard properties of integral  operators with real analytic
kernels and with no singularity imply that the map from $
]-\epsilon^\#_{\tilde{\Omega}},\epsilon^\#_{\tilde{\Omega}}[\times
L^{1}(\partial\Omega)$ to $C^{2}( {\mathrm{cl}}\Omega^\ast)$ which
takes
$(\epsilon,f)$ to the function $\int_{\partial\Omega}R_{q,n}(\epsilon
(t-s))f(s)\,d\sigma_{s}$ of the variable $t \in
{\mathrm{cl}}\Omega^\ast$ is real analytic (cf.~\cite[\S 4]{LaMu11a}).
Then 
by the analyticity of $\Theta^o_j$ and by the continuity of the imbeddings of $C^{0,\alpha}(\partial\Omega)_{0}$ into $L^{1}(\partial\Omega) $ and of $C^{2}(\mathrm{cl}\Omega^\ast)$ into $C^{1,\alpha}(\mathrm{cl}\Omega^\ast)$, we conclude that the map from 
$ ]-\epsilon^\#_{\tilde{\Omega}},\epsilon^\#_{\tilde{\Omega}}[\times \mathcal{U}_{r_{\ast}}$ to $C^{1,\alpha}( {\mathrm{cl}}\Omega^\ast)$  which takes $(\epsilon, \epsilon')$ to the second term in the right hand side of \eqref{est6} is real analytic. Then, by standard calculus in Banach space, we deduce that  $V^{-}_{j,\Omega^\ast}$ is real analytic. By Theorem \ref{thm:lim}, by equality $\Theta^o_j[0,r_\ast]=\tilde{\theta}^o_j$, and by \eqref{rem:lim1}, the validity of \eqref{eq:rep7} follows. Thus the proof is complete.\qquad\end{proof}

\section{A functional analytic representation theorem for the effective conductivity}\label{eff}

In following theorem we answer to the question in \eqref{question}.

\begin{theorem}\label{thm:S}
Let the assumptions of Theorem \ref{thm:Lmbd} hold. Let $k \in \{1,\dots,n\}$. Then there exist $\epsilon_2 \in ]0,\epsilon_1[$ and a real analytic function $\Lambda_{kj}$ from $]-\epsilon_2,\epsilon_2[\times{\mathcal{U}}_{r_{\ast}}$ to $\mathbb{R}$ such that
\begin{equation}\label{eq:S0} 
\begin{split}
\lambda^{\mathrm{eff}}_{kj}[\epsilon]=&\lambda^-\delta_{k,j}+ \epsilon^n
\Lambda_{kj}\Bigl[\epsilon, \frac{\epsilon}{\rho(\epsilon)}\Bigr]\, ,
\end{split}
\end{equation}
for all $\epsilon \in ]0,\epsilon_2[$. Moreover,
\begin{equation}\label{eq:S1}
\begin{split}
\Lambda_{kj}[0,r_{\ast}]=&\frac{1}{|Q|_n}\left(\lambda^+\int_{\partial \Omega}  \tilde{u}_j^+ (t)(\nu_{\Omega}(t))_k\, d\sigma_t-\lambda^-\int_{ \partial \Omega}  \tilde{u}_j^- (t)(\nu_{\Omega}(t))_k\, d\sigma_t\right)\\
&+\frac{|\Omega|_n}{|Q|_n}(\lambda^+-\lambda^-)\delta_{k,j}+\frac{1}{|Q|_n}\int_{\partial\Omega}f(t)t_k\,d\sigma_t\,,
\end{split}
\end{equation}  where $|\Omega|_n$ denotes the $n$-dimensional measure of $\Omega$ and where $\tilde{u}^+_j$, $\tilde{u}^-_j$ are defined as in Theorem \ref{thm:lim}.
\end{theorem}
\begin{proof}Let $U^{+}_j$ be as in Theorem \ref{thm:rep+}. We first note that if $\epsilon \in ]0,\epsilon_1[$, then, by a computation based on the Divergence Theorem, we have
\[
\begin{split}
\int_{\Omega_{p,\epsilon}}\frac{\partial u^+_j[\epsilon](x)}{\partial x_k} \, dx&=\int_{\partial \Omega_{p,\epsilon}}u^+_j[\epsilon](x)(\nu_{\Omega_{p,\epsilon}}(x))_k\, d\sigma_x\\
&=\epsilon^n\int_{\partial \Omega} U^{+}_j\Bigl[\epsilon, \frac{\epsilon}{\rho(\epsilon)}\Bigr](t)(\nu_{\Omega}(t))_k\, d\sigma_t\, .
\end{split}
\]
Then we set
\begin{equation}\label{eq:S2}
\Lambda_{kj}^+[\epsilon,\epsilon']\equiv  \frac{1}{|Q|_n} \int_{\partial \Omega} U^{+}_j[\epsilon,\epsilon'](t)(\nu_{\Omega}(t))_k\, d\sigma_t\, ,
\end{equation}
for all $(\epsilon,\epsilon') \in ]-\epsilon_1,\epsilon_1[\times{\mathcal{U}}_{r_{\ast}}$. Then, by Theorem \ref{thm:rep+}, $\Lambda_{kj}^+$ is a real analytic function from $ ]-\epsilon_1,\epsilon_1[\times{\mathcal{U}}_{r_{\ast}}$ to $\mathbb{R}$. Moreover, by equalities  \eqref{eq:rep2} and \eqref{intnu},
\[
\begin{split}
\Lambda_{kj}^+[0,r_{\ast}]&=\frac{1}{|Q|_n}\int_{\partial \Omega} \tilde{u}_j^+(t)(\nu_{\Omega}(t))_k\, d\sigma_t+\frac{1}{|Q|_n} \int_{\partial\Omega}t_j(\nu_{\Omega}(t))_k\, d\sigma_t\\
&=\frac{1}{|Q|_n} \int_{\partial \Omega} \tilde{u}_j^+(t)(\nu_{\Omega}(t))_k\, d\sigma_t+ \frac{|\Omega|_n }{|Q|_n} \delta_{k,j}\, .
\end{split}
\]

Now let $R>0$ be such that $\mathrm{cl}\Omega \subseteq \mathbb{B}_n(0,R)$ and set $\tilde{\Omega}\equiv \mathbb{B}_n(0,R)\setminus \mathrm{cl}\Omega$. Let $V^{-}_{ j, \tilde{\Omega} }$, $\epsilon^\#_{\tilde{\Omega}}$ be as in Theorem \ref{thm:rep-} (ii). Then a computation based on the Divergence Theorem, on the periodicity of the function which takes $x$ to $u^-_j[\epsilon](x)-x_j$, and on equality \eqref{intnu} shows that
\[
\begin{split}
\int_{Q \setminus \mathrm{cl}\Omega_{p,\epsilon}}& \frac{\partial u^-_j[\epsilon](x)}{\partial x_k} \, dx=\int_{Q \setminus \mathrm{cl}\Omega_{p,\epsilon}} \frac{\partial \bigl(u^-_j[\epsilon](x)-x_j\bigr)}{\partial x_k} \, dx+\delta_{k,j}\bigl(|Q|_n-\epsilon^n |\Omega|_n\bigr)\\
&=\int_{\partial (Q \setminus \mathrm{cl}\Omega_{p,\epsilon})}\bigl(u^-_j[\epsilon](x)-x_j\bigr)(\nu_{Q \setminus \mathrm{cl}\Omega_{p,\epsilon}}(x))_k\, d\sigma_x+\delta_{k,j}\bigl(|Q|_n-\epsilon^n |\Omega|_n\bigr)\\
&=-\int_{\partial  \Omega_{p,\epsilon}}\bigl(u^-_j[\epsilon](x)-x_j\bigr)(\nu_{\Omega_{p,\epsilon}}(x))_k\, d\sigma_x+\delta_{k,j}\bigl(|Q|_n-\epsilon^n |\Omega|_n\bigr)\\
&=-\epsilon^n\int_{\partial \Omega} V^{-}_{ j, \tilde{\Omega} }\Bigl[\epsilon, \frac{\epsilon}{\rho(\epsilon)}\Bigr](t)(\nu_{\Omega}(t))_k\, d\sigma_t+\delta_{k,j}|Q|_n \qquad \forall \epsilon \in ]0,\epsilon^\#_{\tilde{\Omega}}[\, .
\end{split}
\]
Then we set $\epsilon_2\equiv \epsilon^\#_{\tilde{\Omega}}$ and
\[
\Lambda_{kj}^-[\epsilon,\epsilon']\equiv -\frac{1}{|Q|_n} \int_{\partial \Omega} V^{-}_{j, \tilde{\Omega}}[\epsilon,\epsilon'](t)(\nu_{\Omega}(t))_k\, d\sigma_t  \, ,
\]
for all $(\epsilon,\epsilon') \in ]-\epsilon_2,\epsilon_2[\times{\mathcal{U}}_{r_{\ast}}$. By Theorem \ref{thm:rep-} (ii), $\Lambda_{kj}^-$ is a real analytic function from $ ]-\epsilon_2,\epsilon_2[\times{\mathcal{U}}_{r_{\ast}}$ to $\mathbb{R}$. Moreover, by equality \eqref{eq:rep7}, we have
\begin{equation}\label{eq:S3}
\Lambda_{kj}^-[0,r_{\ast}]=-\frac{1}{|Q|_n}\int_{\partial \Omega} \tilde{u}_j^-(t)(\nu_{\Omega}(t))_k\, d\sigma_t-\frac{|\Omega|_n}{|Q|_n}\delta_{k,j}\, .
\end{equation}

Therefore, if we set
\[
\Lambda_{kj}[\epsilon,\epsilon']\equiv \lambda^+ \Lambda_{kj}^+[\epsilon,\epsilon']+\lambda^- \Lambda_{kj}^-[\epsilon,\epsilon']+\frac{1}{|Q|_n}\int_{\partial\Omega}f(t)t_k\,d\sigma_t\, ,
\]
for all $(\epsilon,\epsilon') \in ]-\epsilon_2,\epsilon_2[\times{\mathcal{U}}_{r_{\ast}}$, we deduce that $\Lambda_{kj}$ is a real analytic function from $ ]-\epsilon_2,\epsilon_2[\times{\mathcal{U}}_{r_{\ast}}$ to $\mathbb{R}$ such that equality \eqref{eq:S0} holds for all $\epsilon \in ]0,\epsilon_2[$. Finally, by equalities \eqref{eq:S2}, \eqref{eq:S3}, we deduce the validity of \eqref{eq:S1}.
\qquad\end{proof}

\section{Concluding remarks and extensions}\label{conrem}

By virtue of Theorem \ref{thm:S}, if $\epsilon/\rho(\epsilon)$ has a real analytic continuation around $0$, then the term in the right hand side of equality \eqref{eq:S0} defines a real analytic function of the variable $\epsilon$ in the whole of a neighbourhood of $0$. In particular, we can deduce the existence of $\epsilon_3 \in ]0,\epsilon_2[$ and of a sequence $\{a_i\}_{i=0}^{+\infty}$ of real numbers, such that
\[
\lambda^{\mathrm{eff}}_{kj} [\epsilon]=\lambda^-\delta_{k,j}+ \epsilon^n\Lambda_{kj}[0,r_{\ast}]+\epsilon^{n+1}\sum_{i=0}^{+\infty}a_i \epsilon^i \qquad \forall \epsilon \in ]0,\epsilon_3[\, ,
\]
where the series in the right hand side converges absolutely on $]-\epsilon_3,\epsilon_3[$. Therefore, it is of interest to compute the coefficients $\{a_i\}_{i=0}^{+\infty}$ and this will be the object of future investigations by the authors. We also note that in Dryga{\'s} and Mityushev \cite{DrMi09}, the authors have considered the two-dimensional case with circular inclusions and they have expressed the effective conductivity as a series of the square of the radius $\epsilon$ of the inclusions, under the assumption that, with our notation, $\rho(\epsilon)$ is proportional to $1/\epsilon$ and $f$ and $g$ are equal to $0$. We observe that such an assumption is compatible with condition \eqref{varrhorast}. Hence, in the two-dimensional case, with such a choice of $\rho$, $f$, and $g$, one would try to prove that $\lambda^{\mathrm{eff}}_{kj} [\epsilon]$ can be represented by means of a real analytic function of the variable $\epsilon^2$ (see also Ammari, Kang, and Touibi \cite{AmKaTo05}). If the dimension is greater than or equal to three, instead, one would expect a different behaviour (cf., \textit{e.g.}, McPhedran and McKenzie \cite{McMc78}). Finally, we plan to investigate problem \eqref{bvpe} under assumptions different from \eqref{varrhorast}.

\section*{Acknowledgements} The authors are indebted to Prof.~S.~V.~Rogosin for pointing out problem \eqref{bvpe}. The authors wish to thank Prof.~L.~P.~Castro for several useful discussions.

\end{document}